\pdfoutput=1
\RequirePackage{snapshot}
\documentclass[final, onethmnum, onefignum]{siamltex}

\usepackage[utf8]{inputenc}
\usepackage[T1]{fontenc}

\usepackage{graphicx}
\usepackage{amssymb}
\usepackage{amsmath}
\usepackage{epstopdf}
\usepackage{color}
 
\usepackage[font=footnotesize]{caption}
\usepackage[font=scriptsize]{subcaption}

\DeclareMathOperator*{\argmindisp}{argmin}

\newcommand{\argmin}{{\rm argmin}}

\newcommand{\prox}{{\rm prox}}
\newcommand{\ps}[1]{\left\langle #1 \right\rangle}
\newcommand{\R}{\mathbb{R}}

\newcommand{\dist}{d}

\newcommand{\Pos}{{\mathrm{Pos}}_3}
\newcommand{\trace}{\mathrm{trace}}
\newcommand{\deltaSNR}{\mathrm{\Delta SNR}}

\newcommand{\Tr}{{\rm trace}}

\usepackage[boxed, ruled, lined]{algorithm2e}

\usepackage{tikz}
\usepackage{pgfplots}
\usetikzlibrary{spy}

\usepackage{url}
\usepackage{booktabs}
\def\figspace{\vspace{-0.55cm}}

\title{Total variation regularization for manifold-valued data}
\author{The Author}
\author{Andreas Weinmann\footnotemark[2], Laurent Demaret\footnotemark[2], and Martin Storath\footnotemark[3]}

\date{{ \today}}

\begin{document}
\renewcommand{\thefootnote}{\fnsymbol{footnote}}
\footnotetext[2]{%Research Group Fast Algorithms for Biomedical Imaging, 
% Institute of Computational Biology, 
Helmholtz Center Munich, and Department of Mathematics, Technische Universit\"at M\"unchen (Germany), 
andreas.weinmann@tum.de,
laurent.demaret@helmholtz-muenchen.de}

\footnotetext[3]{{Biomedical Imaging Group, École Polytechnique Fédérale de Lausanne (Switzerland), martin.storath@epfl.ch}}

\maketitle

\begin{abstract}
We consider total variation minimization for manifold valued data.
We propose a cyclic proximal point algorithm and a parallel proximal point algorithm to minimize TV functionals with $\ell^p$-type data terms in the manifold case.
These algorithms are based on iterative geodesic averaging which makes them easily applicable to a large class of data manifolds.
As an application, we consider denoising images which take their values in a manifold.
We apply our algorithms to diffusion tensor images, interferometric SAR images as well as sphere and cylinder valued images.    
For the class of Cartan-Hadamard manifolds (which includes the data space in diffusion tensor imaging) 
we show the convergence of the proposed TV minimizing algorithms to a global minimizer. 
\end{abstract}

% AMS Subject Classification: 65K05, 65K10, 68U10, 94A08

% Keywords: total variation minimization, manifold-valued data, proximal point algorithm, diffusion tensor imaging

%%%%%%%%%
\section{Introduction}
%%%%%%%%%

% !TEX root = tvForManifoldArticle.tex

Data taking values in a manifold appear naturally in various signal and image processing applications. One example is diffusion tensor imaging
where the data live in the Riemannian manifold of positive (definite) matrices; see, e.g., \cite{basser1994mr,pennec2006riemannian}. 
Other examples are color images based on non-flat color models \cite{chan2001total,vese2002numerical,kimmel2002orientation}.
Here the data has circle- or sphere-valued components. Data with values on the circle also appear in the context of interferometric SAR images \cite{massonnet1998radar}.
$SO(3)$ and motion group-valued data were considered in \cite{rahman2005multiscale}.  

Processing manifold valued data has gained a lot of interest in recent years. To mention only some examples,
wavelet-type multiscale transforms for manifold data have been considered in
\cite{rahman2005multiscale,grohs2009interpolatory, weinmann2012interpolatory}. Manifold valued partial differential equations are the subject of study in the papers 
\cite{tschumperle2001diffusion, chefd2004regularizing, grohs2013optimal}. Furthermore, statistical issues on Riemannian manifolds 
are topic of \cite{fletcher2012,fletcher2007riemannian,fletcher2004principal}.  

The present paper deals with total variation (TV) minimization for data taking values in a manifold.
Our main application is the denoising of such manifold-valued data. 
For scalar data, TV minimization was shown to be a powerful tool for edge-preserving denoising \cite{rudin1992nonlinear}. 
For images, the anisotropic version of TV minimization is given by  
\begin{equation} \label{eq:abstract_min_problem_l1_tv_manifold}
 \argmindisp_{x \in M^{n\times m}} \Big\{ \frac{1}{p} \sum_{i,j =1}^{n,m}  d(x_{ij},f_{ij})^p +  \alpha \sum_{i,j=1}^{n-1,m} d(x_{ij},x_{i+1,j}) +  \alpha \sum_{i,j=1}^{n,m-1} d(x_{ij},x_{i,j+1})
 \Big\}.
\end{equation}
For scalar data, the symbol $d(y,z) = |y-z|$ simply denotes the euclidean distance between $y$ and $z.$
The distance to data $f$ is measured in the $\ell^p$ norm. For Gaussian noise, $p=2$ is reasonable, whereas, for 
noise types with heavier tails such as Laplacian noise, $p=1$ is more natural.
The natural generalization of the total variation problem for data on a manifold $M$ is given by
using the distance $d$ induced by the Riemannian metric on the manifold 
instead of the euclidean distance in \eqref{eq:abstract_min_problem_l1_tv_manifold}. 

We introduce algorithms to minimize the $\ell^p$-TV functional for manifold data and
show convergence towards a (global) minimizer for the class of Cartan-Hadamard manifolds where the data space in diffusion tensor imaging is a particular member of.
Our experiments show the denoising capability of $\ell^p$-TV minimization in the manifold context.

\subsection{Total variation regularization for scalar and vector-valued data}

Total variation regularization was first introduced in the early 1990ies by Rudin, Osher and Fatemi \cite{rudin1992nonlinear}.
A central advantage of total variation regularization compared with classical Tykhonov regularization is that it preserves sharp edges \cite{strong2003edge,gousseau2001natural}.
Especially due to this property TV minimization has been used in a vast amount of applications. Examples are 
biomedical imaging \cite{dey2006richardson}, geophysics \cite{anagaw2012edge}
and computer vision \cite{zach2007duality,chen2006total}, to mention only a few.

Theoretical properties of total variation regularization have been investigated in a series of papers. For instance, results on existence and uniqueness of minimizers have been proved in \cite{chambolle1997image}.
Connections to wavelet shrinkage are shown in \cite{petrushev1999nonlinear}  
and equivalences between diffusion techniques, thresholding strategies and TV minimization be found in \cite{steidl2004equivalence}.

A lot of different algorithms for TV minimization of scalar- and vector-valued images have been proposed in the last 20 years. 
In their original work, Rudin, Osher, Fatemi \cite{rudin1992nonlinear} consider $\ell^2$ data terms. 
They use gradient descent on the Euler-Lagrange equations of the (scalar-valued) total variation functional.
Further methods are based on Fenchel duals \cite{chambolle2004algorithm},
the alternating direction method of multipliers  \cite{yang2010fast}, and split Bregman methods \cite{goldstein2009split}.

Several authors have studied the total variation problem with $\ell^1$ data terms \cite{alliney1992digital,nikolova2002minimizers,chan2005aspects}. 
Approaches based on the $\ell^1$-TV functional enjoy the edge preserving properties of TV regularizers while, in addition, being more robust to outliers. 
Various solution strategies have been proposed for the $\ell^1$-TV problem. To mention some examples, schemes based on smooth approximations are presented in \cite{nikolova2004variational, nikolova2008efficient}; semi-smooth Newton method based approaches are the topic of \cite{clason2010semismooth};
primal-dual methods are proposed in \cite{chambolle2011first,dong2009efficient}. 

%%%%%%%%

\subsection{Algorithms for TV minimization for manifold-valued data}

In this paper we derive algorithms for the TV minimization problem
\eqref{eq:abstract_min_problem_l1_tv_manifold} on manifold-valued data. 
Our algorithms are based on iterative geodesic averaging.
More precisely, we decompose the TV functional in \eqref{eq:abstract_min_problem_l1_tv_manifold} into a 
sum of functionals in such a way that we are able to explicitly compute the proximal mappings of these functionals on the manifold. 
We obtain that these proximal mappings are given in terms of points on certain geodesics. So, in order to make the algorithms work on 
a concrete manifold, the only operations we need are those needed for calculating geodesics. 
The spaces which frequently occur as data spaces are matrix groups or related symmetric spaces. 
So usually, there are explicit formulas available for this task.

Our algorithms are iterative schemes. In each iteration,
we apply the above proximal mappings of the functionals decomposing the TV functional.
The first algorithm is a \emph{cyclic proximal point algorithm}. This means that we successively apply 
the proximal mappings using the output related to the $i$th summand as a new input for the proximal mapping
of the $(i$+$1)$th summand. The second algorithm is a \emph{parallel proximal point algorithm}. Here 
the proximal mappings are calculated for the same initial point 
and averaged afterwards.
Since computing mean values on a manifold is a relatively expensive iterative procedure, 
we also consider a variant which only does approximative averaging (but yields comparable results).
We call this variant \emph{fast parallel proximal point algorithm}.
Due to the averaging procedure, the parallel algorithms need more operations in total. 
However, they have higher potential for parallelization.  

Our algorithms belong to the class of proximal splitting methods (for manifold valued data).
A survey on proximal splitting methods (for scalar data) 
is \cite{combettes2011proximal}. In section 7, the paper \cite{combettes2011proximal} also describes a related parallel algorithm.   
Parallel proximal point algorithms were also considered in \cite{boyd2011distributed}.
Cyclical proximal point algorithms have been studied in \cite{Bertsekas2011in} (for linear spaces)
and in \cite{bavcak2013computing} where they are applied for the computation of means and medians in Hadamard spaces.

Principally, our algorithms work for all $\ell^p$ data terms with $p \geq 1$ as well as for regularizing terms based on $q$th variation, $q \geq 1$, instead of total variation.
This in particular includes the classical Tikhonov regularization which corresponds to $p = q = 2.$
For $p,q = 1,2$ we give nice closed form expressions; in the other cases, one needs the numerical solution of a certain nonlinear equation. 
We furthermore consider Huber data terms as well as Huber regularizing terms (which are sometimes called ``Huber-ROF'').
The latter are employed to avoid unwanted staircasing effects; see \cite{chambolle2011first}.
 
For the class of Cartan-Hadamard manifolds, we obtain the convergence of our geodesic averaging based schemes towards a (global) minimizer of \eqref{eq:abstract_min_problem_l1_tv_manifold}.
Cartan-Hadamard manifolds are Riemannian manifolds containing many symmetric spaces such as the data space in diffusion tensor imaging.
Our convergence statements also hold true for the more general class of Hadamard spaces and for regularization based on $q$th variation as well as Huber data and regularization terms.

\subsection{Applications}

We demonstrate the denoising capabilities of our algorithms on various data spaces. 
It is the common observation of all experiments that
TV minimization reliably removes noise from manifold-valued data while preserving edges. 

First, we consider diffusion tensor images. 
Diffusion tensor imaging (DTI) is a technique to quantify non-invasively the diffusional characteristics of a specimen \cite{basser1994mr, johansen2009diffusion}.
Here the underlying data space is the space of positive matrices. According to the model, the diffusivity in direction $v$ is determined by $v^TAv$ 
where $A$ is a positive matrix representing data. 
The space of positive matrices becomes a Cartan-Hadamard manifold when equipped with a suitable Riemannian metric. This means that our algorithm provably converge 
to a minimizer in the DTI setup. We demonstrate the denoising performance
 with a real diffusion tensor image of a human brain and with synthetic data; see the Figures~\ref{fig:dti_camino} and \ref{fig:dti_synthetic}. 

Next, we consider interferometric synthetic aperture radar (InSAR) data.
InSAR is an important airborne imaging modality for geodesy \cite{massonnet1998radar}. 
In our concrete example, 
the InSAR image has the interpretation of a wrapped periodic version of a digital elevation model \cite{rocca1997overview}.
Hence the underlying data space is the sphere $\mathrm{S}^1.$
From the experiment in Figure~\ref{fig:SAR} we see that total variation minimization is capable of removing almost all the noise from the InSAR image.
We further observe that  $\ell^1$ and Huber data terms are slightly more robust to
outliers in the data than the $\ell^2$ data term.

Our third application is image denoising in nonlinear color spaces.
We consider the LCh space which consists of real-valued luminance and chromaticity components L and C as well as a $S^1$ valued hue component h.
Thus the underlying manifold is the cylinder $\mathbb{R}^2 \times S^1.$
We note that, although the underlying manifold $\R^2 \times \mathrm{S}^1$ is a product space, the algorithms cannot be applied separately to the components.
In our experiment we obtain a better reconstruction quality by 
manifold-valued TV minimization in the LCh color space than by classical vectorial TV minimization in the standard RGB space; see Figure~\ref{fig:LCh}.

 We continue with the sphere $\mathrm{S}^2.$ Data with values in $\mathrm{S}^2$ appear in, for example, chromaticity-based image processing
\cite{chan2001total} and as orientation fields of three dimensional images \cite{rezakhaniha2012experimental}.
We here consider a synthetic example on which we impose von Mises-Fisher noise.
Figure~\ref{fig:s2_exp} shows that, also for data with values in the sphere $\mathrm{S}^2,$ the noise is almost perfectly removed and  that the edges are not smoothed out.

We conclude with the rotation group $\mathrm{SO}(3)$ as data spaces. 
Data with values in $\mathrm{SO}(3)$ appear, for example, in the context of 
aircraft orientations \cite{rahman2005multiscale}, protein alignments \cite{green2006bayesian}, and the tracking of 3D rotational data arising in robotics \cite{drummond2002real}.
We apply our methods to a synthetic time series on which we imposed noise based on a matrix Fisher distribution. 
The results confirm their denoising capability and they also reveal that a Huber regularizing term
is less affected by staircasing effects; cf. Figure~\ref{fig:so3}.

\subsection{Organization of the article}
We start out by developing algorithms for TV minimization for manifold valued data in Section~\ref{sec:Algos}. Then we show the convergence
of our algorithms towards (global) minimizers of the TV minimization problem \eqref{eq:abstract_min_problem_l1_tv_manifold} in Hadamard spaces in Section~\ref{sec:CHproofsSec}.
In Section~\ref{sec:Applications}, we apply our algorithms to denoising data on concrete manifolds.

% !TEX root = tvForManifoldArticle.tex

\section{Algorithms for TV minimization for manifold-valued data}
\label{sec:Algos}

In the following we propose two algorithms for total variation minimization for data which take their values in a manifold.
We consider $\ell^1$ and $\ell^2$ as well as Huber data terms. 
Our algorithms are based on iterative geodesic averaging. 
The appearing geodesic averages are the minimizers of certain proximal mappings which arise as follows: 
we split the TV functional into basic building blocks and consider the 
proximal mappings of these building blocks. 
The first algorithm performs the iteration of the proximal mappings in a cyclical way whereas
the second does so in a parallel way.

\subsection{Splitting of the TV functional and proximal mappings}

Let us consider the problem of (bivariate) $\ell^p$-TV$^q$ minimization
\begin{equation} \label{eq:minimization_problem}
  \frac1p \sum_{i,j} d^p(x_{ij},f_{ij}) + \alpha \frac1q \sum_{i,j} d^q(x_{ij},x_{i+1,j})  
+ \alpha \frac1q \sum_{i,j} d^q(x_{ij},x_{i,j+1})  
  \to   {\min}. 
\end{equation}
The data $f_{ij}$ as well as the arguments $x_{ij}$ to minimize take their values in 
a Riemannian manifold $M$. The symbol $d^p$ denotes the $p$th power of the distance induced by the Riemannian metric.
For a bounded (complete) Riemannian manifold, the functional obviously has a minimizer since continuous functions have minima on compact sets.
For the unbounded case, we notice that going too far away from the data $f$ leads to high functional values. Hence the set of minimizer
candidates is actually confined to a bounded set which brings us back to the already discussed situation and minimizers exist.

Setting $q=1$ in  \eqref{eq:minimization_problem}, we get the discrete (anisotropic) TV functional with $\ell^p$ data term. In particular, if $p=1,$ we are in the $\ell^1$-TV setting.
The case $q=2$ corresponds to the classical Tikhonov regularization term in the scalar case. 
We comment on Huber data and regularizing terms in Section~\ref{sec:Huber}.

Our approaches towards the minimization of \eqref{eq:minimization_problem} are based on 
rewriting \eqref{eq:minimization_problem} as a sum of simpler functions. We consider
the function $F: M \times \ldots \times M \to \mathbb{R}$ which is the data term given by 
\begin{equation} \label{def:F}
   F(x) = \frac{1}{p} \sum_{i,j=1}^{n,m} d^p(x_{ij},f_{ij}),
\end{equation}
as well as, for  $i,j,$ the functions $G_{ij},H_{ij}: M \times \hdots \times M \to \mathbb{R}$ given by
\begin{equation}\label{def:G}
\begin{split}
   G_{ij}(x) &=  \frac{1}{q} d^q(x_{ij},x_{i,j+1}), \\
   H_{ij}(x) &=  \frac{1}{q} d^q(x_{ij},x_{i+1,j}). 
   \end{split}
   \end{equation}
Using this notation, the minimization problem \eqref{eq:minimization_problem} has the form
\begin{equation}
\label{eq:minimization_problem2}
 F(x) + \alpha \sum_{i,j} G_{ij}(x) + \alpha \sum_{i,j} H_{ij}(x) \to {\min}.
\end{equation}
For each summand in \eqref{eq:minimization_problem2}, we consider its proximal mapping \cite{moreau1962fonctions, ferreira2002proximal, azagra2005proximal}. 
The proximal mappings of the $G_{ij}$ are defined by the minimization problem
\begin{equation}\label{eq:DefProxy}
\prox_{\lambda G_{ij}} x = \argmin_{y \in M^{n \times m}} \left( \lambda G_{ij}(y) +  \frac12  d^2(x,y) \right),
\end{equation}
where the parameter $\lambda>0$ and the distance $d$ on the product manifold $M^{n \times m}$ is given by
$
    d^2(x,y) = \sum_{i,j=1}^{n,m} d(x_{ij},y_{ij})^2.
$
The proximal mappings of $F$ and the $H_{ij}$ are defined analogously. The crucial point is that, using the splitting \eqref{eq:minimization_problem2},
the proximal mappings of all appearing summands can be explicitly computed as geodesic averages. More precisely, solving the minimization problem
of \eqref{eq:DefProxy} reduces to computing points on shortest geodesics joining given points. The same is true for the analogous problems for $F$ and the $H_{ij}.$

We give a heuristic introductory derivation of these facts now.
A mathematically precise statement is formulated as Proposition~\ref{prop:RiemDesProx} later on.  
We assume for the time being (without mentioning) that the points in the Riemannian
manifold are sufficiently near to each other such that the following arguments apply. 
Let us consider the proximal mapping of $G_{ij}$ given by \eqref{eq:DefProxy}. 
A necessary condition for $x' \in M^{n \times m}$ to be a minimizer in \eqref{eq:DefProxy} is that $0$ is in the (sub)gradient of $\lambda G_{ij}(x') + \tfrac{1}{2}d^2(x,x').$ 
This immediately implies that, for the $(k,l)^{th}$ component of the proximal mapping $\prox_{\lambda G_{ij}} x,$ we have 
\begin{equation} \label{eq:ProxyConst}
   (\prox_{\lambda G_{ij} x})_{kl} =  x_{kl}   \quad \text{ for } \quad k \neq i,i+1 \text{ and } l \neq j.
\end{equation}
For $k=i,i+1$ and $l=j$ we use that the gradient of the mapping $M \to \R,$
$ z  \mapsto  d^p(z,v)/p$ fulfills (see, e.g., \cite[Eq. (2.8)]{tron2012riemannian})
\begin{equation*} 
\nabla_z d^p(z,v)/p = \frac{\exp^{-1}_z (v)}{d^{2-p}(z,v)}.
\end{equation*}
Hence, we get as necessary conditions for a minimizer $x'$ in \eqref{eq:DefProxy} that 
\begin{align*}
\lambda \tfrac{1} {d^{2-p}(x'_{ij},x'_{i,j+1})}  \exp^{-1}_{x'_{ij}} (x'_{i,j+1})  &+  \exp^{-1}_{x'_{ij}}x_{ij} = 0, \notag \\
\lambda \tfrac{1}{d^{2-p}(x'_{ij},x'_{i,j+1})} \exp^{-1}_{x'_{i,j+1}} (x'_{ij})  &+  \exp^{-1}_{x'_{i,j+1}}x_{i,j+1} = 0.
\end{align*} 
Here $\exp_z^{-1}$ denotes the inverse of the Riemannian exponential mapping at the point $z.$ Thus, both summands in the first
condition are tangent vectors at $x'_{ij}.$ They point in opposite directions, and so, 
the first condition implies that the three points $x'_{i,j},$ $x'_{i,j+1},$ and $x_{i,j}$ lie on a common geodesic in $M.$
Analogously, the second condition implies that the three points $x'_{i,j},$ $x'_{i,j+1},$ and $x_{i,j+1}$ also lie on a common geodesic.
Hence, the four points must lie on one geodesic. In particular, the points $x'_{i,j},$ $x'_{i,j+1}$ -- 
which are the $(i,j)^{th}$ and $(i,j+1)^{th}$ component of the proximal mapping of $G_{ij}$ applied to $x$ -- lie on the geodesic joining 
$x_{i,j}$ and $x_{i,j+1}.$

Then, after some technical considerations (cf.~the proof of Theorem~\ref{thm:ConvergenceAlgA}), 
the location of the points $x'_{i,j} = (\prox_{\lambda G_{ij}} x)_{i,j}$ and $x'_{i,j+1}=(\prox_{\lambda G_{ij}} x)_{i,j+1}$  
are explicitly given as follows. We have  
\begin{equation}\label{eq:ProxOfGij2}
\begin{split}
   (\prox_{\lambda G_{ij}} x)_{ij} &=  [x_{ij},x_{i,j+1}]_t,  \\
   (\prox_{\lambda G_{ij}} x)_{i,j+1} &=  [x_{i,j+1},x_{ij}]_t,
   \end{split}
\end{equation}
where the symbol $[\cdot,\cdot]_t$ denotes the point reached after time $t$ on the unit speed geodesic starting at the first argument in direction of the second argument. 
In the TV case $(q=1)$,
\begin{equation}  \label{eq:ProxAreGeodesic}
  t = \begin{cases}   \lambda,                \qquad &\mbox{if }  \lambda < \tfrac{1}{2} d(x_{ij},x_{i,j+1}),  
                   \\ d(x_{ij},x_{i,j+1})/2,  \qquad       &\mbox{else.}  
      \end{cases}             
\end{equation}
For $q=2,$ which corresponds to quadratic variation, we get
\begin{equation}  \label{eq:proxGell2}
t =   \frac{\lambda}{1+ 2 \lambda} d(x_{ij},x_{i,j+1}).
\end{equation} 

The proximal mappings of the $H_{ij}$ are obtained in a completely analogous way. It remains to find the proximal mapping of $F$ which means finding the proximal mapping of the distance function
in $M$. This is well known and can be found, e.g., in \cite{ferreira2002proximal}. They can again be written as geodesic averages and are explicitly given by  
\begin{align} \label{eq:ProxOfF}
(\prox_{\lambda F})_{ij}(x) =  [x_{ij},f_{ij}]_t,    
\end{align}
where, for the $\ell^2$ data term,  
\begin{align} \label{eq:ProxOfF1}
     t= \tfrac{\lambda}{1+\lambda} d(x_{ij},f_{ij}).  
\end{align}
For the $\ell^1$ data term,
\begin{align} \label{eq:ProxOfF2}
  t= \begin{cases}
    \lambda,   &\mbox{ if } \lambda < d(x_{ij},f_{ij}) ,\\
     d(x_{ij},f_{ij}),            &\mbox{ else.}
    \end{cases} 
\end{align}
This corresponds to the equivalent of {\em soft thresholding} in the context of mani\-folds.

A mathematically precise formulation of the results of the above heuristic derivation is as follows; its proof is provided in Section~\ref{sec:CHproofsSec} below the proof of Theorem~\ref{thm:ConvergenceAlgA}.
\begin{proposition}
\label{prop:RiemDesProx}
Let $M$ be a complete connected Riemannian manifold. Then a solution $y^\ast$ of the proximal mapping related minimization problem
in \eqref{eq:DefProxy} is given as follows.
For $ \quad k \neq i,i+1 \text{ and } l \neq j,$ define $y^\ast_{k,l} =  x_{k,l};$
the component $y^\ast_{i,j}$ is chosen as the point on a shortest geodesic between $x_{ij}$ and $x_{i,j+1}$ starting from $x_{ij}$ reached after time $t$ with $t$ given by \eqref{eq:ProxAreGeodesic},\eqref{eq:proxGell2};  the component $y^\ast_{i,j}$ is obtained analogously, but starting at $x_{i,j+1}$ on the same geodesic with reverse direction.
Analogous statements hold for the $H_{ij}$ and $F.$ 

In particular, if there is only one shortest geodesic (which is always the case for nearby points),
then the proximal mappings of $F,G_{ij}, H_{ij}$ are well defined and they are given by 
\eqref{eq:ProxOfF},\eqref{eq:ProxyConst},\eqref{eq:ProxOfGij2} and their analogues for $H_{ij}.$  
\end{proposition}

\subsection{A cyclic proximal point algorithm for TV minimization for mani\-fold-valued data}

The first algorithm we propose for TV minimization for manifold-valued data is a cyclic proximal point algorithm based on geodesic averaging.
For vector space data, cyclical proximal point algorithms were considered in \cite{Bertsekas2011in}. For Hadamard spaces, they were 
investigated by M. Ba{\v{c}}{\'a}k \cite{bavcak2013computing} who applied them to the computation of means and medians. 

We now derive a cyclic proximal point algorithm for the minimization of the $\ell^p$-$\mathrm{TV}^q$ functional \eqref{eq:minimization_problem}.
We consider the problem in the form $F(x) +$ $\alpha \sum_{i,j} G_{ij}(x)$ $+ \alpha \sum_{i,j} H_{ij}(x)$ given by \eqref{eq:minimization_problem2}.
We first apply the proximal mapping of $F$ which is given as pointwise geodesic averages of data $f_{ij}$ and the argument of the functional $x_{ij};$ see \eqref{eq:ProxOfF}. 
Then we successively apply the proximal mappings of all the $G_{ij}.$ They are given by \eqref{eq:ProxyConst} and \eqref{eq:ProxOfGij2}
which is again based on geodesic averaging.
As a last step, the analogous operations are executed for the $H_{ij.}$ 

Iteration of all these steps yields the algorithm which is stated as Algorithm~\ref{alg:cyclic_splitting}. 
During the iteration, the parameter $\lambda_n$ of the proximal mappings is successively decreased. 
In this way, the penalty for deviation from the previous iterate is successively increased.
It is chosen in a way such that 
the sequence $\lambda_n$ is square-summable but not summable.  
This is moderate enough not to prevent convergence towards a minimizer; cf. Theorem~\ref{thm:ConvergenceAlgA}.

\begin{algorithm}[ht]
\small
	 \KwIn{Manifold-valued image $f \in M^{n \times m}$, regularization parameter $\alpha>0$, parameter sequence for the proximal mappings $\lambda = (\lambda_1, \ldots)$ $\in \ell^2 \setminus \ell^1$.} 
	 \KwOut{ Minimizer  $x$ of the $\ell^p$-TV$^q$ problem \eqref{eq:minimization_problem}. }
	 \Begin{
	 $x \leftarrow  f$\;
\For{$r\leftarrow 1,2,...$}{
  \For {i $\leftarrow$ 1,...,n; j $\leftarrow$ 1,...,m}{                                	   
	       $t \leftarrow \mathrm{calc\_t}_{F}(\lambda_r,p,q, x_{ij},f_{ij})$      \tcc*{{\scriptsize Calculate $t,$ see Table~\ref{tab:waylength_data}.}}
	       $x_{ij} \leftarrow [x_{ij},f_{ij}]_{t} $   	          \tcc*{{\scriptsize Proximal mapping of $F.$}}
  }
  
	\For{i $\leftarrow$ 1,...,n; j $\leftarrow$ 1,...,m-1}{
	       $t \leftarrow \mathrm{calc\_t}_{GH}(\lambda_r\alpha, p,q, x_{ij},x_{i,j+1})$        \tcc*{{\scriptsize Calculate $t,$ see Table~\ref{tab:waylength_reg}.}}                        	   
	       $x'_{ij} \leftarrow [x_{ij},x_{i,j+1}]_{t}$     \tcc*{{\scriptsize Proximal mapping of $G_{ij}.$}}
	       $x'_{i,j+1} \leftarrow [x_{i,j+1},x_{ij}]_{t};$ \newline  
         $x_{ij} \leftarrow x'_{ij};$ $x_{i,j+1} \leftarrow x'_{i,j+1};$                                    	   
	}
	
	\For{i $\leftarrow$ 1,...,n-1; j $\leftarrow$ 1,...,m}{                   
	       $t \leftarrow \mathrm{calc\_t}_{GH}(\lambda_r\alpha,p,q, x_{ij},x_{i,j+1})$          \tcc*{{\scriptsize Calculate $t,$ see Table~\ref{tab:waylength_reg}.}}     	   
	       $x'_{ij} \leftarrow [x_{ij},x_{i+1,j}]_{t}$\tcc*{{\scriptsize Proximal mapping of $H_{ij}.$}}
	       $x'_{i+1,j} \leftarrow [x_{i+1,j},x_{ij}]_{t}; $    \newline
	       $x_{ij} \leftarrow x'_{ij};$ $x_{i+1,j} \leftarrow x'_{i+1,j};$                                     	   
	} 
}
	}
	\caption{Cyclic proximal point algorithm for $\ell^p$-TV$^q$ for manifold data.}
	\label{alg:cyclic_splitting}
 \end{algorithm}

\subsection{A parallel proximal point algorithm for TV minimization for mani\-fold-valued data}
\label{sec:ParallelProxPoint}
Parallel proximal point algorithms were, for example, 
considered in \cite{boyd2011distributed,combettes2011proximal,setzer2012vector}. 
We here apply a related parallel proximal point algorithm to TV minimization for manifold-valued data.  
A great advantage of this approach is its immediate parallelizability.

As for the cyclic algorithm, we split the TV functional into a sum of simpler functionals.
But instead of applying the proximal mappings successively ,
we apply all the proximal mappings to the same initial data. Then the results of the
different proximal mappings are averaged. 

In order to split the TV functional, we consider the mappings
\begin{align}
  G_e &=  \sum_{j:j \text{ even}} \sum_i  \frac{1}{q} d^q(x_{ij},x_{i,j+1}),  \notag \\
  G_o &=  \sum_{j:j \text{ odd}} \sum_i  \frac{1}{q} d^q(x_{ij},x_{i,j+1}). \label{eq:DefGeGo}
\end{align}
The mapping $H_e,H_o$ are defined analogously, exchanging the role of $i$ and $j$.
Then we have that the $\ell^p$-$\mathrm{TV}^q$ functional can be decomposed into $F + G_e + G_o$ $+ H_e + H_o.$
Since $G_e = \sum_{j:j \text{ even} }$ $\sum_i G_{ij}$, the proximal mapping of $G_e$ is explicitly given by
\begin{equation} \label{eq:ProxGeGo}
(\prox_{\lambda G_e} x)_{i,j} = \begin{cases}
                                     [x_{ij},x_{i,j+1}]_{t_1},      \quad &j \text{ even,}     \\
                                       [x_{ij},x_{i,j-1}]_{t_2},    \quad &j \text{ odd. }
                                \end{cases}       
\end{equation}
Here $t_1$ and $t_2$ are defined by \eqref{eq:ProxAreGeodesic},\eqref{eq:proxGell2} (cf.\ the derivation of \eqref{eq:ProxOfGij2}). 
The proximal mapping of $G_o$ is obtained by exchanging the terms ``even'' and ``odd'' in the above formula.
For $H_e,H_o,$ one exchanges the roles of $i$ and $j.$ 

\newcommand{\mean}{{\rm mean}}
Equipped with these explicit formulas for the proximal mappings, the next step is to average the results of the application of the proximal mappings.
Since our data live in a Riemannian manifold, the usual arithmetic mean in a vector space is not available. However, it is well known 
(cf. \cite{karcher1977riemannian,kendall1990probability, pennec2006riemannian, fletcher2007riemannian})
that 
\begin{equation} \label{eq:DefMean}
   z^\ast = \argmin_{z \in M} \sum_{i=1}^N d(z,z_i)^2
\end{equation}
is the appropriate definition of the mean $z^\ast= \mean(z_1,\ldots,z_N)$  of the $N$ elements $z_i$ on the manifold $M$. 
The mean is in general not globally defined since the minimization problem has no unique solution in general.
For the $z_i$  being in a small ball, however, it is unique. The size of the ball depends on the sectional curvature of the manifold $M.$  
Details and further information can, e.g., be found in \cite{kendall1990probability, karcher1977riemannian}.  

In order to get the mean in the product manifold $M^{n \times m}$ we just have to compute the component-wise means. Applied to the above proximal mapping,
we get one iteration of the parallel algorithm at pixel $i,j$ by
\begin{align}
       &x^{(k+1)}_{ij} =    \\
        \mean &([x^{(k)}_{i,j},x^{(k)}_{i,j+1}]_{t_1},              
                               [x^{(k)}_{i,j},x^{(k)}_{i,j-1}]_{t_2},
                               [x^{(k)}_{i,j},x^{(k)}_{i+1,j}]_{t_3},
                               [x^{(k)}_{i,j},x^{(k)}_{i-1,j}]_{t_4},
                               [x^{(k)}_{i,j},f_{i,j}]_{t_5} )   \notag
\end{align}
where the $t_i$ are computed according to \eqref{eq:ProxOfF1}, \eqref{eq:ProxOfF2} and \eqref{eq:ProxAreGeodesic}, \eqref{eq:proxGell2}, respectively. 
So the iterate at pixel $(i,j)$ is obtained by the mean of geodesic averages of the old iterate in a neighborhood.
The whole algorithm is summed up as Algorithm~\ref{alg:parallel_splitting}.

\newcommand{\approxmean}{{\rm approx\_mean}}

\begin{algorithm}[ht]
\small
	 \KwIn{ Manifold-valued image $f \in M^{n \times m}$, regularization parameter $\alpha>0$, parameter sequence for the proximal mappings $\lambda = (\lambda_1, \ldots)$ $\in \ell^2 \setminus \ell^1$.} 
	 \KwOut{ Minimizer  $x$ of the $\ell^p$-TV$^q$ problem \eqref{eq:minimization_problem}. }
	 \Begin{
	 $x \leftarrow  f$\;
%	 $\mu \leftarrow \mu_0$\;
%	 $\lambda \leftarrow 0$\;
%\For{ reached stopping criterion}{
\For {$r \leftarrow 1,2,...$}{
  \For {$i\leftarrow 1,...,n; j \leftarrow 1,...,m$}{       
	       $t \leftarrow \mathrm{calc\_t}_{F}(\lambda_r, p,q, x_{i,j},f_{i,j})$          \tcc*{{\scriptsize Calculate $t,$ see Table~\ref{tab:waylength_data}.}}
	       $z^{(1)} \leftarrow [x_{ij},f_{ij}]_{t} $   	          \tcc*{{\scriptsize Proximal mapping of $F.$}}
         $t \leftarrow \mathrm{calc\_t}_{GH}(\lambda_r \alpha,p,q, x_{i,j},x_{i,j+1})$ \tcc*{{\scriptsize Calculate $t,$ see Table~\ref{tab:waylength_reg}.}}
         $z^{(2)} \leftarrow [x_{ij},x_{i,j+1}]_{t} $   	          \tcc*{{\scriptsize Proximal mapping of $G_e/G_o.$}}
         $t \leftarrow \mathrm{calc\_t}_{GH}(\lambda_r \alpha,p,q, x_{i,j},x_{i,j-1})$ \tcc*{{\scriptsize Calculate t by  \eqref{eq:ProxAreGeodesic},\eqref{eq:proxGell2}}}
         $z^{(3)} \leftarrow [x_{ij},x_{i,j-1}]_{t} $   	          \tcc*{{\scriptsize Proximal mapping of $G_o/G_e.$}}
         $t \leftarrow \mathrm{calc\_t}_{GH}(\lambda_r \alpha,p,q, x_{i,j},x_{i+1,j})$ \tcc*{{\scriptsize Calculate $t,$ see Table~\ref{tab:waylength_reg}.}}
         $z^{(4)} \leftarrow [x_{ij},x_{i+1,j}]_{t} $   	          \tcc*{{\scriptsize Proximal mapping of $H_o/H_e.$}}
         $t \leftarrow \mathrm{calc\_t}_{GH}(\lambda_r \alpha,p,q, x_{i,j},x_{i-1,j})$ \tcc*{{\scriptsize Calculate $t,$ see Table~\ref{tab:waylength_reg}.}}
         $z^{(5)} \leftarrow [x_{ij},x_{i-1,j}]_{t} $   	          \tcc*{{\scriptsize Proximal mapping of $H_o/H_e.$}}
         $x_{ij} \leftarrow \mean(z^{(1)},z^{(2)},z^{(3)},z^{(4)},z^{(5)})$  \tcc*{{\scriptsize Intrinsic mean.}}
         Alternative:
         $x_{ij} \leftarrow \approxmean(z^{(1)},z^{(2)},z^{(3)},z^{(4)},z^{(5)})$  \tcc*{{\scriptsize  Fast \newline \hfill approximative variant (cf. \eqref{eq:approx_mean_Formula})}.} 
  }
}  
  	}
	\caption{Parallel proximal point algorithm for $\ell^p$-TV$^q$ for manifold data.}
	\label{alg:parallel_splitting}
 \end{algorithm}

In contrast to the euclidean case there is no closed form expression of the intrinsic mean defined by \eqref{eq:DefMean} in Riemannian manifolds.
The methods used for computing the mean are of iterative nature and are thus more time consuming. 
Perhaps the most well known method for computing the intrinsic mean, is the gradient descent already mentioned in Karcher's seminal paper \cite{karcher1977riemannian}; cf. also \cite{fletcher2007riemannian}, for example.
The iteration for computing the intrinsic mean of the points $x_1,\ldots,x_N$ is given by
\begin{equation} \label{eq:GradientDescentIntMean}
     x^{(k+1)} =  \exp_{x^{(k)}} \sum_{i=1}^N \tfrac{1}{N} \exp^{-1}_{x^{(k)}}x_i.
\end{equation}
Also approaches based on Newton's method can be found in the literature; see, e.g., \cite{ferreira2013newton}. 

However, it is reported in the literature and also confirmed by the authors' experience that the gradient descent converges rather fast; 
in most cases, $5$-$10$ iterations are enough for five points.
This might explain why this simple method of gradient descent is widely used.

\subsection{Speed-up of the parallel proximal point algorithm} 
\label{sec:SpeedUp}

In Algorithm~\ref{alg:parallel_splitting}, %computing the intrinsic mean is the most time consuming part
we calculate the intrinsic mean of the five points $z^{(i)}$ in the inner loop. 
Each step of the gradient descent \eqref{eq:GradientDescentIntMean} in the computation of the mean of the $z^{(i)}$
takes about half of the time needed for computing the points $z^{(i)}$ by geodesic averaging.
Since we typically need at least five iterations for the gradient descent, computing the means is the most time consuming part of
Algorithm~\ref{alg:parallel_splitting}.

In order to reduce the computation time, we propose to replace the mean by another construction (known as geodesic analogues in the subdivision context \cite{wallner2005convergence}) which is computationally less demanding. 
This construction approximates the mean,
the results are comparable (cf. Figure~\ref{fig:dti_synthetic}), and we can also show convergence towards a minimizer (cf. Theorem~\ref{eq:ConvergenceTheoremParaGeod}).

In order to explain the construction, we rewrite the euclidean  mean $x$ of $n$ points $x_1, \ldots,x_n$ as iterative convex combinations of only two points (in a binary tree like fashion):
\newcommand{\conv}{{\rm conv}}
\begin{equation}
 \label{eq:GeoAvAbstr}
   x = \sum \tfrac{1}{n} x_i  =     \conv_{t_1}(\conv_{t_2}(\ldots,\ldots \conv_{t_l}(x_{i_l},x_{j_l}) \ldots) ,\conv_{t_3}(\ldots,\ldots)).
\end{equation} 
Here, we use the notation $\conv(y,z)_t$ for the convex combination $(1-t)y+tz$ of points $y,z.$ 
For example, for $n=5$ points, we have the following representation
\begin{equation}\label{eq:Setof54geoAn}
   x= \conv_{0.2}(\conv_{0.5}(\conv_{0.5}(x_1,x_2), \conv_{0.5}(x_3,x_4) ) ,x_5).
\end{equation} 
We note that the above representation is not unique.  
The idea is now to replace each euclidean convex combination $\conv_{t}(x,y)$ in \eqref{eq:GeoAvAbstr} by the corresponding Riemannian one, 
i.e., the point $[x,y]_{td(x,y)}$ on the geodesic joining $x$ and $y.$    Then, \eqref{eq:GeoAvAbstr} reads
\begin{equation}
 \label{eq:GeoAvARiem}
   x =  [[\ldots,\ldots [x_{i_l},x_{j_l}]_{t'_l} \ldots]_{t'_2} ,[\ldots,\ldots]_{t'_3}]_{t'_1},
\end{equation}
where $t'_k = t_k d_k$ and $d_k$ denotes the distance of the elements in the bracket the $d_k$ is attached to. (This technicality arises since we consider unit speed geodesics.) 
We consider the above decomposition \eqref{eq:Setof54geoAn} and transport it to the Riemannian setting \eqref{eq:GeoAvARiem}. 
Then, instead of using the mean in Algorithm~\ref{alg:parallel_splitting} we propose to use the alternative procedure (called ``approx$\_$mean'' 
in Algorithm~\ref{alg:parallel_splitting}) given by
\begin{equation}
 \label{eq:approx_mean_Formula}
    x  =  [[[z^{(1)},z^{(2)}]_{0.5 d_1},[z^{(3)},z^{(4)}]_{0.5 d_2}]_{0.5 d_3},z^{(5)}]_{0.2 d_4}.        
\end{equation}
Here each $d_k$ again denotes the distance of the elements in the bracket the $d_i$ is attached to.  
The points $z^{(i)}$ are the results of the application of the proximal mappings 
in Algorithm~\ref{alg:parallel_splitting}.

The full algorithm is given by Algorithm~\ref{alg:parallel_splitting} using the part referred to as ``Alternative''.

\subsection{Huber regularizing and data terms}  \label{sec:Huber}

\begin{table}
\centering
\begin{tabular}{p{0.15\textwidth}p{0.75\textwidth}}
\toprule
\textbf{Regularizer} 	& {\bf Geodesic path length} (value of  $\mathrm{calc\_t}_{GH}(\lambda, p,q, y,z)$)\\ \toprule
  TV 	& $\min(\lambda, d(y,z)/2)$ \\ \midrule
  TV$^2$ &	 $\frac{\lambda}{1+ 2\lambda} d(y,z)$ \\ \midrule
  Huber 	&  $\begin{cases}
                            \frac{2 \lambda \tau^2}{1+ 4 \lambda \tau^2} d(y,z),    &\text{if }d(y,z) < \frac{\omega(1+ 4 \lambda \tau^2)}{\sqrt{2}\tau} ,\\ 
                            \min \left( d(y,z)/2, \sqrt{2}\lambda \omega \tau \right),     & \text{otherwise.}                             
                        \end{cases}$   \\
\bottomrule
\end{tabular}
\caption{Geodesic path length for proximal point problems
 associated to $G$ and $H$
 for the regularization terms considered in this article.
  }
  \label{tab:waylength_reg}
\end{table}

Total variation regularized images may suffer from the undesirable creation of steps in the solution. This is often called staircasing effect. 
An effective way to decrease staircasing is to replace the total variation by the \emph{Huber regularizer} sometimes called \emph{Huber-ROF model} \cite{chambolle2011first}.
To this end, we replace $d^q$ in the definition of the TV$^q$ functional \eqref{eq:minimization_problem} by $h \circ d$ which is the concatenation of the distance on the manifold and the \emph{Huber function} $h.$ 
The Huber function $h$ is defined, for $s>0$, by
\begin{equation}  \label{eq:huber}
h(s) = \begin{cases}
       \tau^2  s^2,     &\text{for }  s < \omega / (\sqrt{2}\tau) ,\\ 
       \omega \sqrt{2}\tau s - \omega^2/2,   & \text{otherwise,}
       \end{cases}
       \qquad \tau,\omega > 0.
\end{equation}
It is a square function (for small arguments) smoothly glued with an absolute value function (for large arguments).  
In analogy to \eqref{def:G} and \eqref{eq:minimization_problem2}  we write the Huber regularizer as
$\sum_{i,j} G^h_{ij}$ $+ \sum_{i,j} H^h_{ij}$ with 
$G^h_{ij}(x)$ $=  h \circ d(x_{ij},x_{i,j+1})$ and 
$H^h_{ij}(x)$ $=  h \circ d(x_{ij},x_{i+1,j}).$
As in \eqref{eq:ProxyConst} and \eqref{eq:ProxOfGij2}, the proximal mappings of the $G^h_{ij}$ are given by (cf. the proof of Theorem~\ref{thm:ConvergenceAlgA})
\begin{equation}\label{eq:ProxOfGijHuber}
   (\prox_{\lambda G^h_{ij}} x)_{kl} = 
   \begin{cases}
     [x_{ij},x_{i,j+1}]_t, & k = i \text{ and } l = j, \\
 [x_{i,j+1},x_{ij}]_t, &  k = i \text{ and } l = j + 1, \\
     x_{kl},  & \text{otherwise,}
   \end{cases}
\end{equation}
where
\begin{align} \label{eq:ProxOfHuberRegT}
 t= \begin{cases}
                            \frac{2 \lambda \tau^2}{1+ 4 \lambda \tau^2} d(x_{i,j},x_{i,j+1}),    &\text{if }d(x_{ij},x_{i,j+1}) < \frac{\omega(1+ 4 \lambda \tau^2)}{\sqrt{2}\tau},\\ 
                            \min \left( d(x_{ij},x_{i,j+1})/2, \sqrt{2}\lambda \omega \tau \right),     & \text{otherwise.}                             
                        \end{cases}   
\end{align}
So replacing the procedure to calculate the geodesic length $t$ in Algorithm~\ref{alg:cyclic_splitting} and Algorithm~\ref{alg:parallel_splitting}
by the procedure to calculate $t$ given by \eqref{eq:ProxOfHuberRegT} yields cyclic and parallel minimization algorithms with the Huber regularizing term.

We can use the Huber function for the data term, i.e., we let
$ F_h(x,f) = \sum\nolimits_{i,j}h \circ d(x_{ij},f_{ij}).$  For small distances the Huber data term behaves like the $\ell^2$ data term but it is more robust to outliers.
The proximal mapping of $F_h$ is given by (cf. the proof of Theorem~\ref{thm:ConvergenceAlgA})
\begin{align} \label{eq:ProxOfHuber}
&(\prox_{\lambda F_h})_{ij}(x) =  [x_{ij},f_{ij}]_t,  \\ \notag      
&\text{ where } t= \begin{cases}
                            \frac{2 \lambda \tau^2}{1+ 2 \lambda \tau^2} d(x_{ij},f_{ij}),    &\text{if }d(x_{ij},f_{ij}) < \frac{\omega(1+ 2 \lambda \tau^2)}{\sqrt{2}\tau} ,\\ 
                            \min \left( d(x_{ij},f_{ij}), \sqrt{2}\lambda \omega \tau \right),    & \text{otherwise.}                             
                        \end{cases}   
\end{align}
Using the proximal mapping of the Huber data term \eqref{eq:ProxOfHuber} instead of the proximal mapping of $F$ in 
Algorithm~\ref{alg:cyclic_splitting} and Algorithm~\ref{alg:parallel_splitting}, 
respectively, yields cyclic and parallel TV minimization algorithms with the Huber data term.

\begin{table}
\centering
\begin{tabular}{p{0.15\textwidth}p{0.75\textwidth}l}
\toprule
\textbf{Data term} 	& {\bf Geodesic path length} (value of ${\rm calc\_t}_{F}(\lambda, p,q, y,z))$  \\ \toprule
  $\ell^1$ 	& $\min(\lambda, d(y,z))$ \\ \midrule
  $\ell^2$	& $\frac{\lambda}{1+ \lambda} d(y,z)$ \\ \midrule
  Huber 	& $\begin{cases}\frac{2 \lambda \tau^2}{1+ 2 \lambda \tau^2} d(y,z),    &\text{if }d(y,z) < \frac{\omega(1+ 2 \lambda \tau^2)}{\sqrt{2}\tau},\\ 
                            \min{\left( d(y,z), \sqrt{2}\lambda \omega \tau \right)},     & \text{otherwise.}                             
                        \end{cases} $ \\
\bottomrule
\end{tabular}
\caption{Geodesic path lengths for the proximal point problem
 associated to $F$ for the data terms considered in this article.
  }
    \label{tab:waylength_data}
\end{table}

\section{Convergence in Hadamard spaces}
\label{sec:CHproofsSec}

In this section we show the convergence of Algorithm~\ref{alg:cyclic_splitting} and Algorithm~\ref{alg:parallel_splitting} for a certain class of spaces
which the TV functionals we consider are convex on.

For general Riemannian manifolds, the $\ell^p$-TV$^q$ functional \eqref{eq:minimization_problem} is not necessarily convex. The perhaps simplest example where convexity fails is
the one-dimensional sphere $S^1$ (cf. \cite{strekalovskiy2011total}). 
In the non-convex case, the study of (global) convergence becomes much more involved and is out of the scope of this paper.   
Here, we treat the quite large class of Cartan-Hadamard manifolds where we still have convexity.
These are complete Riemannian manifolds of nonpositive sectional curvature.  
Prominent examples are the spaces of positive matrices (which are the data space in diffusion tensor imaging) and the hyperbolic spaces.
For details we refer to \cite{doCarmo1992ri} or to \cite{ballmann1985manifolds}.

The proofs in this section work in the more general setup of Hadamard spaces without additional effort. 
Hadamard spaces are certain metric spaces generalizing the
concept of Cartan-Hadamard manifolds which are particular instances. Examples of Hadamard spaces which are not Cartan-Hadamard manifolds are 
the metric trees of \cite{Sturm2003he}. Since there is no additional effort, we consider Hadamard spaces
as  underlying spaces in this section.

A Hadamard space is a geodesic space, i.e., for each two points $x,y$ the length of the shortest arc connecting $x$ and $y$ (which exists by definition) equals the distance of the points.
Furthermore, there is a certain condition ensuring that the geodesic triangles are ``not fat'' (Triangles on the sphere are ``fat''.). This condition
providing the essential properties of Cartan-Hadamard manifolds in this metric space setting is as follows. For given $x_0,x_1$ there is a point $y$ on the geodesic joining them such that
for every $z,$  $d^2(z, y)$ $\leq \tfrac{1}{2} d^2 (z, x_0)$  $+ \tfrac{1}{2} d^2 (z, x_1)$ $- \tfrac{1}{4} d^2 (x_0, x_1)$. 
For details we refer to \cite{Sturm2003he} and the references therein or to the book \cite{bridson1999metric}.

We note that the $\ell^p$-TV$^q$ functionals given by \eqref{eq:minimization_problem} are convex in a Hadamard space. 
This is because the distance function is doubly convex in Hadamard spaces; see \cite{Sturm2003he}.

Our first goal is to show the convergence of Algorithm~\ref{alg:cyclic_splitting} which is the geodesic averaging algorithm based on cyclical application of 
proximal mappings.   

\begin{theorem} \label{thm:ConvergenceAlgA}
For data in a (locally compact) Hadamard space Algorithm~\ref{alg:cyclic_splitting} converges towards a minimizer of the $\ell^p$-TV$^q$ functional.
The statement remains true when using Huber data and regularizing terms based on \eqref{eq:huber}.
\end{theorem}

\begin{proof}
We first show that, in a Hadamard space, the proximal mappings of the functions $F,G_{ij}$ and $H_{ij}$ 
are given by \eqref{eq:ProxOfF},\eqref{eq:ProxyConst},\eqref{eq:ProxOfGij2} and their analogues for $H_{ij}.$
We also show that the proximal mappings of the Huber regularizing and data terms are given by \eqref{eq:ProxOfGijHuber}, \eqref{eq:ProxOfHuberRegT}
and \eqref{eq:ProxOfHuber}, respectively.
We start with the mappings $G_{ij}.$  
The proximal mapping of $G_{ij}$ is given by 
\begin{equation} \label{eq:defproxproof}
\prox_{\lambda G_{ij}} x = \argmin_y   \lambda \tfrac{1}{q} d(y_{ij},y_{i+1,j})^q + \tfrac{1}{2} \sum_{k,l} d(y_{kl},x_{kl})^2.
\end{equation}
Hence, every minimizer $y^\ast$ must fulfill $y^\ast_{kl} = x_{kl}$ for $k \neq i$ and $l \neq j,j+1.$ 
Otherwise, letting $y^\ast_{kl} = x_{kl}$ for $k \neq i$ and $l \neq j,j+1$ would decrease the functional value which contradicts the minimizer property. 
This shows \eqref{eq:ProxyConst}.

Now let $y^\ast$ be a minimizer of \eqref{eq:defproxproof}.
We show that the four points $x_{ij},x_{i,j+1},y^\ast_{ij},$ $y^\ast_{i,j+1}$ lie on one geodesic. 
We may assume that $d(x_{ij}, y^\ast_{ij})$ $\leq d(x_{ij}, x_{i,j+1})$ since, otherwise, setting $y'_{ij} = y'_{i,j+1} = x_{i,j+1}$
would yield a lower functional value in \eqref{eq:defproxproof}. By the same argument, we may assume
that $d(x_{i,j+1}, y^\ast_{i,j+1}) \leq d(x_{ij}, x_{i,j+1}).$

We define the point $z= [x_{ij},x_{i,j+1}]_{t_1}$ as the point reached on the unit speed geodesic starting at $x_{ij}$ after time $t_1 = d(x_{ij},y_{ij}^\ast).$ 
Analogously, we let $z'= [x_{i,j+1},$ $x_{i,j+1}]_{t_2}$ be the point on the same geodesic when starting from $x_{k+1}$ after time $t = d(x_{k+1},y_{k}^\ast).$  

We first consider the case with ordering $x_{ij}, z, z', x_{i,j+1}$ 
when running on the geodesic starting at $x_k$ (including the case $z=z'$.)
Since these points lie on a geodesic, we have that 
\begin{align} \label{eq:EstimGeodesics}
d(x_{ij},x_{i,j+1}) &= d(x_{ij}, z ) + d(z,z') +d(z',x_{i,j+1})  \\
                    &\leq d(x_{ij},y^\ast_{ij}) + d(y^\ast_{ij},y^\ast_{i,j+1}) +d(y^\ast_{i,j+1},x_{i,j+1}) \notag 
\end{align}
Here, the inequality is true since every geodesic is a shortest path in a Hadamard space.
By our choice of $t_1,t_2,$ this implies 
$ d(z,z') \leq  d(y^\ast_{ij},y^\ast_{i,j+1})$.
As a consequence, the functional value $a(z,z') \leq a(y^\ast_{ij},y^\ast_{i,j+1})$ where  
\begin{equation} \label{eq:reducedMinproblem}
  a(v,v') = \tfrac{1}{2}d(v,x_{ij})^2 + \lambda \tfrac{1}{q} d(v,v')^q + \tfrac{1}{2}d(v',x_{i,j+1})^2.
\end{equation}
This is the essential part of the functional in \eqref{eq:defproxproof} meaning that
minimizing the functional in \eqref{eq:defproxproof} is equivalent to minimizing $a.$
Hence, since geodesics are unique in a Hadamard space, $z= y^\ast_{ij}$ and $z'= y^\ast_{i,j+1}$.
Thus, these four points lie on a geodesic.

Next, we consider the case with ordering $x_{ij},$ $z',$ $z,$ $x_{i,j+1}$ 
when running on the geodesic starting at $x_k.$ Then we have that $a(z,z) \leq a(z,z'),$ since 
$d(z,x_{ij})$ $< d(z,x_{i,j+1}).$ Hence, we obtain a lower functional value which means that
this situation cannot occur for a minimizer of \eqref{eq:defproxproof}.

Summing up, we know that the points $x_{ij}, y^\ast_{ij},y^\ast_{i,j+1}, x_{i,j+1}$ lie on a geodesic in this ordering.

Next, we need the precise position of $y^\ast_{ij},y^\ast_{i,j+1}$ on the geodesic.   
We consider the real numbers $d=d(x_{ij},x_{i,j+1})$ and the time points $t_1$ and $t_2$ given above.
Minimization of $a$ is now equivalent to minimizing, for $0 \leq t_1,t_2 \leq d/2,$ 
$$
  a'(t_1,t_2) = \tfrac{1}{2} t_1^2 + \lambda \tfrac{1}{q}(d-t_1-t_2)^q + \tfrac{1}{2} t_2^2.
$$
By symmetry and uniqueness, a minimizer fulfills $t_1=t_2.$ Hence, we have to find a minimizer of   
$$
  a''(t) = t^2 + \tfrac{\lambda}{q}(d-2t)^q .
$$
For $q=1,$ we get the solution 
$$
    t = \min(\lambda,d/2 ),
$$
and for $q = 2,$
$$
    t = \frac{\lambda d}{2+ 2 \lambda}.
$$
This shows \eqref{eq:ProxOfGij2} and the subsequent formulas \eqref{eq:ProxAreGeodesic} and \eqref{eq:proxGell2}.
The corresponding proof for the $H_{ij}$ is analogous.

Next, we consider the Huber data term $F_h$ based on \eqref{eq:huber}. We show that its proximal mapping is given by 
\eqref{eq:ProxOfHuber}. Similar as above, we define the point $z = [x_{ij},f_{ij}]_{t},$ where $t=d(x_{ij},y^\ast_{ij})$
and $y^\ast_{ij}$ is the $(i,j)^{th}$ component of the proximal mapping of $\lambda F_h$ at $x.$ Modifying the
arguments above, we see that $z = y^{\ast}_{ij}.$ Thus the three points $x_{ij},y^{\ast}_{ij},f_{ij}$ 
lie on one geodesic and it remains to determine $t.$ 
This leads to minimizing the scalar problem $ t \mapsto 2 \lambda h(d-t)$ $+t^2$ under the constraint $0\leq t \leq d,$ where we let $d=d(x_{ij},f_{ij}).$  
To solve this problem one applies an analogous calculation as the one in Example 4.5 of \cite{Chaux2007Var} and concludes \eqref{eq:ProxOfHuber}. 
The corresponding proof of \eqref{eq:ProxOfF} for the $\ell^p$ type data term $F$ is analogous.

It remains to consider the Huber regularizer $G^h$+$H^h$ from Section~\ref{sec:Huber}. 
We proceed analogous to the proof for the regularizer $G$+$H$ to obtain that the four points $x_{ij}, y^\ast_{ij},y^\ast_{i,j+1}, x_{i,j+1}$ lie on a geodesic in this ordering
which shows \eqref{eq:ProxOfGijHuber}.
Then proceeding as in the proof for the Huber data term $F_h$ with a similar calculation as in \cite{Chaux2007Var} we obtain the formula 
\eqref{eq:ProxOfHuberRegT}.

Since Algorithm~\ref{alg:cyclic_splitting} only produces convex combinations of the points involved, we have that the iterates produced
by the algorithm stay in the convex hull of the data $(f_{i,j})_{i,j}.$ Since all functions $F,G_{ij},H_{ij}$ are continuous,
they are Lipschitz on that convex hull. Hence the assumptions of \cite[Theorem 3.4]{bavcak2013computing} are fulfilled, and the application of this theorem
yields the convergence of Algorithm~\ref{alg:cyclic_splitting}.
\end{proof}

Next, using the above techniques, we supply the proof of Proposition~\ref{prop:RiemDesProx} which was a statement on proximal mappings in the setup of Riemannian manifolds (not Hadamard spaces).

{\em Proof of Proposition~\ref{prop:RiemDesProx}.}
We show the statement for the functions $G_{ij}.$ 
As in the proof of Theorem~\ref{thm:ConvergenceAlgA} we see that $y^\ast_{k,l} =  x_{k,l}$ for $k \neq i,i+1$  and $l \neq j.$
Then we consider the components $y^\ast_{ij}$ and $y^\ast_{i,j+1}$ of a minimizer.
We fix a shortest geodesic connecting $x_{ij}$ and $x_{i,j+1}.$ 
We define the points $z,z'$ on this geodesic as in the proof of Theorem~\ref{thm:ConvergenceAlgA}. 
Then we see that we may apply the same arguments as in that proof to reduce to the situation
where $x_{ij}, z', z, x_{i,j+1}$ lie on the geodesic in this ordering.  
Then the estimate \eqref{eq:EstimGeodesics} is true since we chose a shortest geodesic.
This implies $ d(z,z') \leq  d(y^\ast_{ij},y^\ast_{i,j+1})$.
On the other hand, $y^\ast$ is a minimizer of $G_{ij}$. This implies 
$ d(y^\ast_{ij},y^\ast_{i,j+1}) \leq  d(z,z').$ 
Hence, equality holds and $y^\ast_{ij}$ and $y^\ast_{i,j+1}$ lie on a shortest geodesic connecting $x_{ij}$ and $x_{i,j+1}.$
The statements for the functionals $H_{ij}$ and $F$ follow analogously. 
$\Box$

The goal of the rest of this section is to show that Algorithm~\ref{alg:parallel_splitting} and its fast variant (with the approximate mean calculation from \eqref{eq:GeoAvARiem})
converge in a Hadamard space. To this end, we first show a generic convergence statement for parallel proximal point algorithms. 
\begin{theorem} \label{thm:convergenceParallel}
We consider a convex function $g$ defined on a Hadamard space which has a minimizer.
Let $g=g_1+\ldots+g_n$ and assume that all summands are convex and lower semicontinuous. 
Assume further that the positive parameter sequence $\lambda = (\lambda_1,\ldots)$ is square-summable but not summable. 
We consider the iteration 
\begin{equation}\label{eq:ConvParallel} 
x^{k+1} = \mean \left( \prox_{\lambda_k g_1} x^k,\ldots, \prox_{\lambda_k g_n} x^k  \right).   
\end{equation}
Here $\mean$ is the intrinsic mean in the Hadamard space defined by \eqref{eq:DefMean}.
If there is a constant $L>0$ such that, for all $g_i$ and all $k$,
\begin{equation}\label{eq:ConditionConvParallel} 
  g_i(x^k) - g_i(\prox_{\lambda_k g_1} x^k) \leq L  \cdot d(x^k,\prox_{\lambda_k g_1} x^k).   
\end{equation}
then the iteration \eqref{eq:ConvParallel} converges to a minimizer of $g$. 
\end{theorem}

The proof of this statement is an adaption of the proof of Theorem 3.4 in \cite{bavcak2013computing} to the parallel setting.
In \cite{bavcak2013computing}, the applied method of proof is addressed to \cite{Bertsekas2011in}. We need the following two lemmas which are 
Lemma 2.6 and Lemma 3.2 of \cite{bavcak2013computing}. 
  
\begin{lemma}\label{lem:Bac26}
 Let $a_k,$ $b_k,$ $c_k$ be sequences of positive numbers. Assume that $\sum c_k < \infty$ and that, for all $k$,
 $a_{k+1} < a_k - b_k + c_k.$ Then the sequence $a_k$ converges and $\sum b_k < \infty.$  
\end{lemma}

\begin{lemma}\label{lem:Bac32}
 Consider a convex and lower semicontinuous function $h$ on a (locally compact) Hadamard space. Then,  
 \begin{align}\label{lem:BacEstimate}
  h(\prox_{\lambda h} x) - h (y)  \leq  \tfrac{1}{2 \lambda}   \left( d(x,y)^2  -  d( \prox_{\lambda h} x  ,y)^2 \right),
 \end{align}
 for any $y$ in the Hadamard space.
\end{lemma}

Equipped with these preparations, we show Theorem~\ref{thm:convergenceParallel}.

{\em Proof of Theorem~\ref{thm:convergenceParallel}.} 
The function $x \mapsto d(x,y)^2$ is uniformly convex (cf. \cite{Sturm2003he}). Thus, using Jensen's inequality (cf. \cite{Sturm2003he}), 
we get, for the intrinsic mean in the Hadamard space $x^{k+1}$, that
\begin{equation}
    d(x^{k+1},y)^2 \leq \sum_{i=1}^n \tfrac{1}{n} d(\prox_{\lambda_k g_i} x^k,y)^2, \label{eq:UseJensen}
\end{equation}
for all $y.$ In the following, we use the notation $ \tilde{x}^{k+1}_i = \prox_{\lambda_k g_i} x^k$
for the proximal mapping of $g_i$ at the previous iterate $x^k.$
Using Lemma~\ref{lem:Bac32}, we estimate 
\begin{align}
   d( \tilde{x}^{k+1}_i  ,y)^2 &\leq d(x^k,y)^2 - 2 \lambda_k \left( g_i(\tilde{x}^{k+1}_i)-g_i(y)    \right) \label{eq:UseBac32}   \\
   &= d(x^k,y)^2 - 2 \lambda_k \left( g_i(x^k) -g_i(y) \right) + 2 \lambda_k \left( g_i(x^k) - g_i(\tilde{x}^{k+1}_i) \right). \notag
\end{align}
We combine the estimates \eqref{eq:UseJensen} and \eqref{eq:UseBac32} to obtain
\begin{align}
d(&x^{k+1},y)^2  \label{eq:EstAa} \\
&\leq \frac{1}{n} \sum_{i=1}^n d(x^k,y)^2 - \frac{2 \lambda_k}{n} \sum_{i=1}^n \left( g_i(x_n)-g_i(y) \right) + \frac{2 \lambda_k}{n} \sum_{i=1}^n  \left(g_i(x_n)-g_i(\tilde{x}^{k+1}_i) \right) \notag  \\
& = d(x^k,y)^2 - \frac{2 \lambda_k}{n} \left( g(x_n)-g(y) \right) + \frac{2 \lambda_k}{n} \sum_{i=1}^n  \left(g_i(x_n)-g_i(\tilde{x}^{k+1}_i) \right).   \notag
\end{align}
The next goal is to estimate the last summand on the right hand side. To this end, we use that, by assumption,
$$
  g_i(x^k) - g_i(\tilde{x}^{k+1}_i) \leq L   d(x^k,\tilde{x}^{k+1}_i).   
$$ 
Furthermore, since $\tilde{x}^{k+1}_i$ minimizes the expression in the definition of the proximal mapping, we obtain
$$
  g_i(\tilde{x}^{k+1}_i) + \tfrac{1}{2 \lambda_k} d(x^k,\tilde{x}^{k+1}_i) \leq g_i(x^k).
$$ 
Applying these estimates successively yields 
\begin{align}
g_i(x^k) - g_i(\tilde{x}^{k+1}_i) &\leq L   d(x^k,\tilde{x}^{k+1}_i)   \label{eq:EstLastTR}  \\
                                  &\leq L 2 \lambda_k \frac{g_i(x^k)- g_i(\tilde{x}^{k+1}_i)}{ d(x^k,\tilde{x}^{k+1}_i)} \leq 2 \lambda_k L^2. \notag
\end{align}
This allows us to estimate the last summand on the right-hand side of \eqref{eq:EstAa} by  
\begin{align*}
\frac{2 \lambda_k}{n} \sum_{i=1}^n  \left(g_i(x_n)-g_i(\tilde{x}^{k+1}_i) \right) \leq \frac{2 \lambda_k}{n} \sum_{i=1}^n 2 \lambda_k L^2 = 4 \lambda_k^2 L^2.   
\end{align*}
Thus, \eqref{eq:EstAa} now reads 
\begin{align}  \label{eq:ABCest}
    d(&x^{k+1},y)^2  \leq  d(x^k,y)^2 - \frac{2 \lambda_k}{n} \left( g(x_n)-g(y) \right) + 4 \lambda_k^2 L^2.
\end{align}

Next, we consider a minimizer $y^{\ast}$ of $g$ and plug it into \eqref{eq:ABCest} above. Then $g(x_n)-g(y^{\ast})\geq 0,$ and we
may apply Lemma~\ref{lem:Bac26} to \eqref{eq:ABCest}. This yields that $d(x^k,y^{\ast})$ converges as $n \to \infty$ and that 
\begin{equation}  \label{eq:Res26App}
  \sum_{n \in \mathbb{N}} \frac{2 \lambda_k}{n} \left( g(x_n)-g(y) \right)    < \infty.
\end{equation}
Since the parameter sequence $\lambda=(\lambda_1,\ldots)$ is not summable, \eqref{eq:Res26App} implies that $g(x^{k_l}) \to g(y^{\ast})$ on a subsequence $k_l.$
Since $d(x^k,y^{\ast})$ is bounded and the underlying space is locally compact, we may choose another subsequence $k_{l_r}$
such that $x^{k}$ converges on this subsequence; call the limit $x^{\ast}.$ Then by the lower semicontinuity of $g,$ we have $g(x^{\ast}) \leq g(y^{\ast}).$  
Then, since $y^{\ast}$ is a minimizer, also $x^{\ast}$ is a minimizer. We apply Lemma~\ref{lem:Bac26} for a second time, but now for $y=x^{\ast}.$
As a result $d(x^k,x^\ast)$ converges. Moreover, $d(x^k,x^\ast) \to 0,$ since this is true on a subsequence.
This completes the proof. 
\quad $\Box$

For the fast variant of the parallel proximal algorithm introduced in Section~\ref{sec:SpeedUp} we replaced the intrinsic mean by the approximation \eqref{eq:GeoAvARiem}.
In order to obtain convergence of the corresponding algorithm, we need the following result.

\begin{theorem} \label{thm:convergenceParallelFast}
  The statement of Theorem~\ref{thm:convergenceParallel} remains true if we replace the intrinsic mean \eqref{eq:DefMean} by its approximation \eqref{eq:GeoAvARiem}.
\end{theorem}

\begin{proof}
We show that \eqref{eq:UseJensen} remains true if we replace the intrinsic mean by the construction given in \eqref{eq:GeoAvARiem}.
By the convexity of the function $a(z) = d(z,y)^2$ we have that $a([z,y]_{td(z,y)}) \leq (1-t) a(z) + t a(y)$ for all $z,y \in M$. 
We successively apply this inequality to every geodesic average in \eqref{eq:GeoAvARiem} (in a top-down fashion). 
As a first step, we  obtain
\begin{equation*}
   a(x) \leq (1-{t'_1}) a([\ldots,\ldots [x_{i_l},x_{j_l}]_{t'_l} \ldots]_{t'_2})  + {t'_1} a([\ldots,\ldots]_{t'_3}).
\end{equation*} 
Proceeding further, we get 
\begin{equation} \label{eq:i22}
   a(x) \leq  c_1 a(x_1) + \ldots + c_n a(x_n),
\end{equation} 
where the $c_k$ are products with factors $t'_r$ and $(1-t'_r).$ Reversing the construction of the $t'_r$ in \eqref{eq:GeoAvAbstr},
we see that, for each $x_k$, the factor $c_k$ equals $1/n.$ Plugging the definition $a(z)=d(z,y)^2$ and $c_k = 1/n$
in \eqref{eq:i22} yields \eqref{eq:UseJensen} for the construction \eqref{eq:GeoAvARiem}.
Then we can follow the rest of the proof of Theorem~\ref{thm:convergenceParallel} to conclude the assertion of the present theorem.
\end{proof}

Our main results concerning the convergence of the parallel proximal algorithm are as follows.
\begin{theorem} \label{eq:ConvergenceTheoremParaGeod}
  The parallel proximal algorithm for $\ell^p$-TV$^q$ minimization (Algorithm~\ref{alg:parallel_splitting}) and its approximative variant 
  converge towards a minimizer in every (locally compact) Hadamard space. 
  The statement remains true when using Huber regularization and Huber data terms based on \eqref{eq:huber}.
\end{theorem}

\begin{proof}
As a first step, we have to show that the proximal mappings of the functions $G_e,G_o,H_e,H_o$ given in
and below \eqref{eq:DefGeGo} are in fact given by \eqref{eq:ProxGeGo} and the explanations following \eqref{eq:ProxGeGo}.
We only consider $G_e$ since the other cases are analogous.  
We use the fact that $G_e$ is a sum of the functions $G_{ij}.$ 
More precisely, $G_e = \sum_{j:j \text{ even} }$ $\sum_i G_{ij}.$
Considering this form of $G_e,$ we see that the minimization problem in the definition of the proximal mapping separates 
into minimization problems which require minimizing expressions of the form \eqref {eq:reducedMinproblem}. 
Hence, the $(i,j)^{th}$ component of the proximal mapping of $G_e$ equals the corresponding component of the proximal mapping of $G_{ij}.$
This proximal mapping has been considered in the proof of Theorem~\ref{thm:ConvergenceAlgA}.
Its $(i,j)^{th}$ component agrees with the expression in \eqref{eq:ProxGeGo}. This shows \eqref{eq:ProxGeGo}.

For the Huber regularizing term, we consider the functionals $G^h_e,G^h_o,H^h_e,H^h_o$ defined in analogy to $G_e,G_h,H_e,H_o$ by replacing $d^q$ by
$h \circ d$ where $d$ is the Huber function \eqref{eq:huber} in \eqref{eq:DefGeGo}. Then following the argument for $G_e$ in the previous
paragraph, we see that the $(i,j)^{th}$ component of the proximal mappings of $G_e^h$ agrees with the corresponding component of the proximal mapping of $G^h_{ij}$ 
given by \eqref{eq:ProxOfGijHuber},\eqref{eq:ProxOfHuberRegT}. Analogous statements are true for $G^h_o,H^h_e,H^h_o$.
The proximal mappings of the $\ell^p$ type data term $F$ and the Huber data term $F_h$ have been shown to agree with \eqref{eq:ProxOfF} and \eqref{eq:ProxOfHuber}, respectively, in the proof 
of Theorem~\ref{thm:ConvergenceAlgA}.  

The next step is to apply Theorem~\ref{thm:convergenceParallel} and Theorem~\ref{thm:convergenceParallelFast} (for the approximative variant).
Since the algorithm only produces convex combinations and intrinsic means, the iterates produced
by the algorithm stay in the convex hull of the data $(f_{i,j})_{i,j}.$ So the involved functions (which are all continuous) are Lipschitz on this convex
hull which means that \eqref{eq:ConditionConvParallel} is fulfilled. Hence, we may apply Theorem~\ref{thm:convergenceParallel} and Theorem~\ref{thm:convergenceParallelFast}
and conclude the assertion of the theorem.
\end{proof}

%%%%%%%%%t
\section{Applications}
%%%%%%%%%
\label{sec:Applications}

In this section, we apply the algorithms proposed in this paper to concrete manifolds which frequently occur in applications.
The manifolds we consider are the space of positive matrices $\Pos,$ the spheres $S^1$ and $S^2,$ the product space $S^1 \times \R^2$ (which appears in the context of nonlinear color models)
as well as the rotation group.

In order to make 
Algorithm~\ref{alg:cyclic_splitting} and Algorithm~\ref{alg:parallel_splitting} work in a specific manifold we have to compute geodesics and distances on this manifolds.
This is accomplished using the Riemannian exponential mapping and its inverse. 
Recall that the exponential mapping $\exp_{a}: T_a M \to M$ returns the point $\exp_{a}v$ on the manifold which we obtain 
when following the unit speed geodesic starting at $a$ into the direction of the given tangent vector $v$ for time $\|v\|_a$. 
Conversely, the inverse of the exponential mapping $\exp^{-1}_{a}:M \to T_a M$
gives us the tangent vector $\exp^{-1}_{a}b$ at the point $a$ 
which leads to the point $b$ when following the geodesic with respect to this tangent vector for time $\| \exp^{-1}_{a}b \|_a.$   
Using these mappings, the point $[a,b]_t$ reached on the unit speed geodesic 
joining $a$ and $b$ after time $t$ is given by 
\begin{equation}\label{eq:geodesic_app}
	[a,b]_t =  \exp_{a}(t \cdot \exp^{-1}_a(b) ).
\end{equation}
In order to calculate the geodesic path length $t,$
we further have to calculate distances on the manifold under consideration (cf.~Table~\ref{tab:waylength_reg} and Table~\ref{tab:waylength_data}).
To this end, we use that the distance between points $a$ and $b$ is given by the length of the tangent vector $\exp^{-1}_a(b),$ 
i.e.,
\begin{equation}\label{eq:dist_app}
	d(a,b) = \| \exp^{-1}_a(b) \|_a.
\end{equation}
Here the length is measured with respect to the Riemannian metric in the tangent space of $a$. 
Hence, in order to apply our algorithms for a specific data space,
we only need to instantiate the exponential mapping and its inverse for the corresponding manifold. 
For the data spaces considered in this article, the exponential mappings and their inverse have closed expressions
involving only basic arithmetic operations such as trigonometric functions or matrix exponentials.

The numerical experiments were conducted on a Macbook using a single core of a  2.6 GHz Intel Core i7 processor. 
(Parallelized implementations of our algorithms are out of the scope of this article.)
For the experiments in Figure \ref{fig:LCh},  we optimized the total variation parameter $\alpha$ with respect to the peak signal-to-noise ratio.
In the other experiments, $\alpha$ was determined empirically.
A simple choice for the sequence $\lambda_r$ is $\lambda_r = cr^{-\omega}$ with $c >0.$ The sequence fulfills the condition to be in $ \ell^2 \setminus \ell^1$
for each $0.5 <\omega  \leq 1.$ We here used $\omega = 0.95$ and $c = 3.$
We observed only little differences when using different parameter pairs.
In order to quantitatively assess the denoising performance of total variation regularization on manifolds
we use the \emph{signal-to-noise ratio improvement} (compare \cite[p. 244]{unser2013introduction}).
We consider a manifold-valued version of the signal-to-noise ratio improvement which is given by
\[
	\deltaSNR = 10 \log_{10} \left(  \frac{\sum_{ij} d(g_{ij}, f_{ij})^2   }{\sum_{ij} d(g_{ij}, x_{ij})^2}\right).
\]
Here $f$ is the noisy data, $g$ the ground truth, and $x$ the regularized  restoration.

%%%%%%%%%
\subsection{The space of positive matrices $\Pos$ -- Diffusion tensor imaging}
%%%%%%%%%

Diffusion tensor imaging (DTI) is a non-invasive imaging modality based on nuclear magnetic resonance.
It allows to quantify the diffusional characteristics of a specimen \cite{basser1994mr, johansen2009diffusion}.
Applications are the determination of fiber tract orientations \cite{basser1994mr} and the detection of brain ischemia \cite{le2001diffusion}.
Denoising is an important topic in DTI which has been addressed in various articles, e.g.  \cite{chen2005noise, pennec2006riemannian, basu2006rician}. 

In DTI, the diffusivity of water molecules is captured by a \emph{diffusion tensor,} i.e.,
a (symmetric) positive (definite) $3 \times 3$ matrix $S(p)$ sitting at pixel $p$. 
It is reasonable to consider the space of diffusion tensors $\Pos$ as Riemannian manifold
with the Riemannian metric
$$
g_D(W,V) = \trace(D^{-\tfrac{1}{2}} W D^{-1} V D^{-\tfrac{1}{2}});
$$
see \cite{pennec2006riemannian}. 
Here the symmetric matrices $W,V$ represent tangent vectors in the point $D.$
Equipped with this Riemannian metric the space of positive matrices becomes a Cartan-Hadamard manifold.
Hence, by virtue of Theorem~\ref{thm:ConvergenceAlgA} and Theorem~\ref{eq:ConvergenceTheoremParaGeod}, 
the cyclic proximal point algorithm and both variants of the parallel  algorithm converge to a global minimizer.

For the space of positive matrices, the Riemannian exponential mapping $\exp_D$ is given by
\[
	\exp_D(W) = D^{\frac12} \exp(D^{-\frac12} W D^{-\frac12}) D^{\frac12}.
\]
Here $D$ is a positive matrix and the symmetric matrix $W$ represents a tangent vector in $D.$
The mapping $\exp$ is the matrix exponential.
Furthermore, there is also a closed form expression for the inverse of the Riemannian exponential mapping:
we have, for positive matrices $D,E,$
\[
	\exp^{-1}_D(E) = D^{\frac12} \log(D^{-\frac12} E D^{-\frac12}) D^{\frac12}.
\]
The matrix logarithm $\log$ is well-defined since the argument is a positive matrix.
The matrix exponential and logarithm can be efficiently computed by diagonalizing 
the symmetric matrix under consideration and then applying the scalar exponential and logarithm functions
to the eigenvalues.
The distance between $D$ and $E$ is just the length of the tangent vector $\exp^{-1}_D(E)$ which can be explicitly 
calculated by
\[
	\dist^2(D, E) = \sum_{l=1}^3 \log(\kappa_l)^2,
\]
where $\kappa_l$ is the $l^{\mathrm{th}}$ eigenvalue of the 
matrix $D^{-\frac12} E D^{-\frac12}.$

\begin{figure}
\tikzstyle{myspy}=[spy using outlines={green,lens={scale=4},width=0.98\textwidth, height=0.7\textwidth, connect spies, every spy on node/.append style={ultra thick}}]
\centering
\def\figfolder{./experiments/ex10/}
\def\figwidth{0.48\columnwidth}
\def\nodeSpy{(-1.2, 1)}
\def\nodeWindow{(3.0,-5.0)}
\begin{subfigure}[t]{\figwidth}
	\begin{tikzpicture}[myspy]
	\node {\includegraphics[interpolate=true,width=\textwidth]{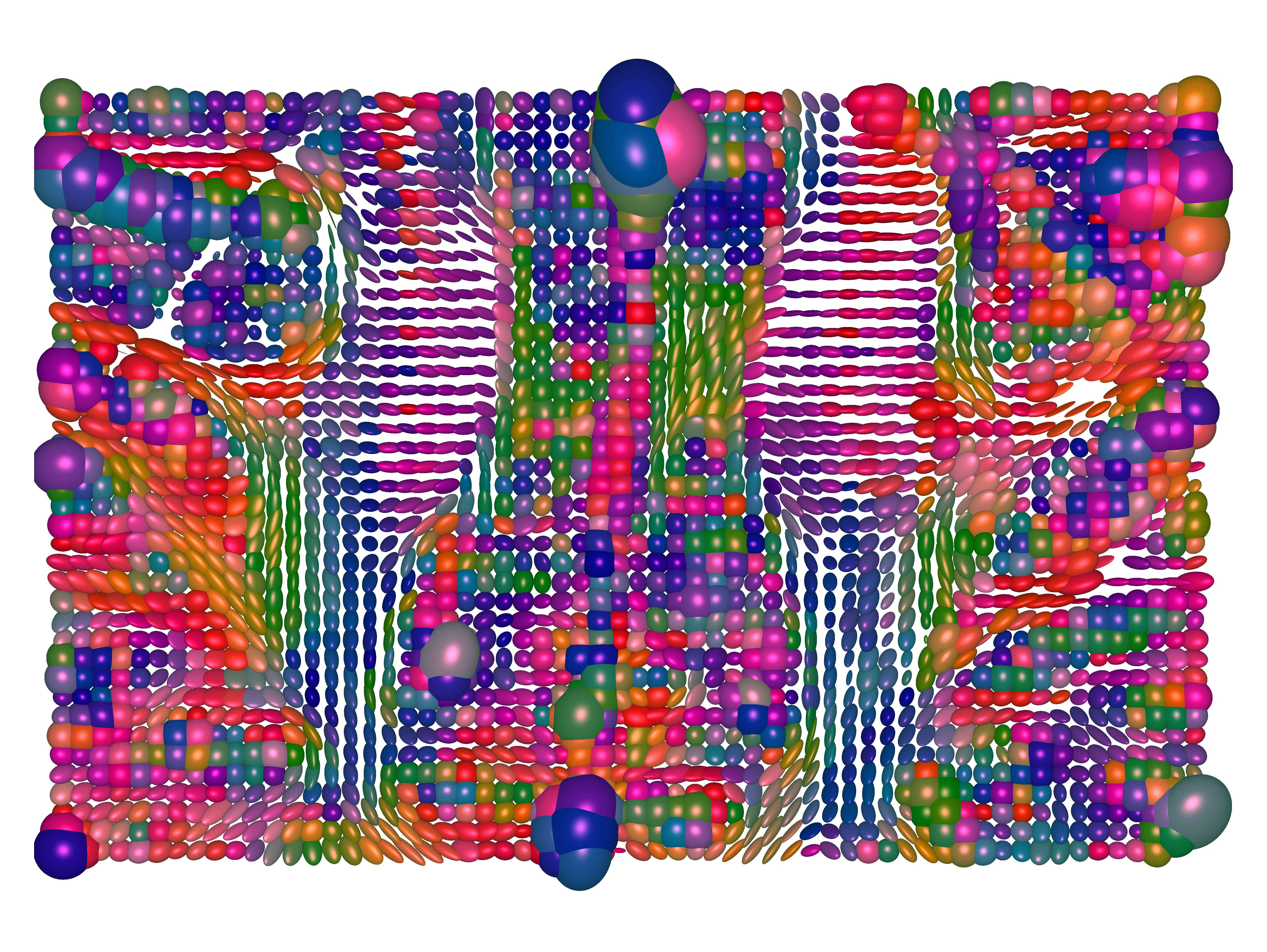}};
	\spy on \nodeSpy in node [left] at \nodeWindow;
\end{tikzpicture}
\end{subfigure}
\hfill
\begin{subfigure}[t]{\figwidth}
	\begin{tikzpicture}[myspy]
	\node {\includegraphics[interpolate=true,width=\textwidth]{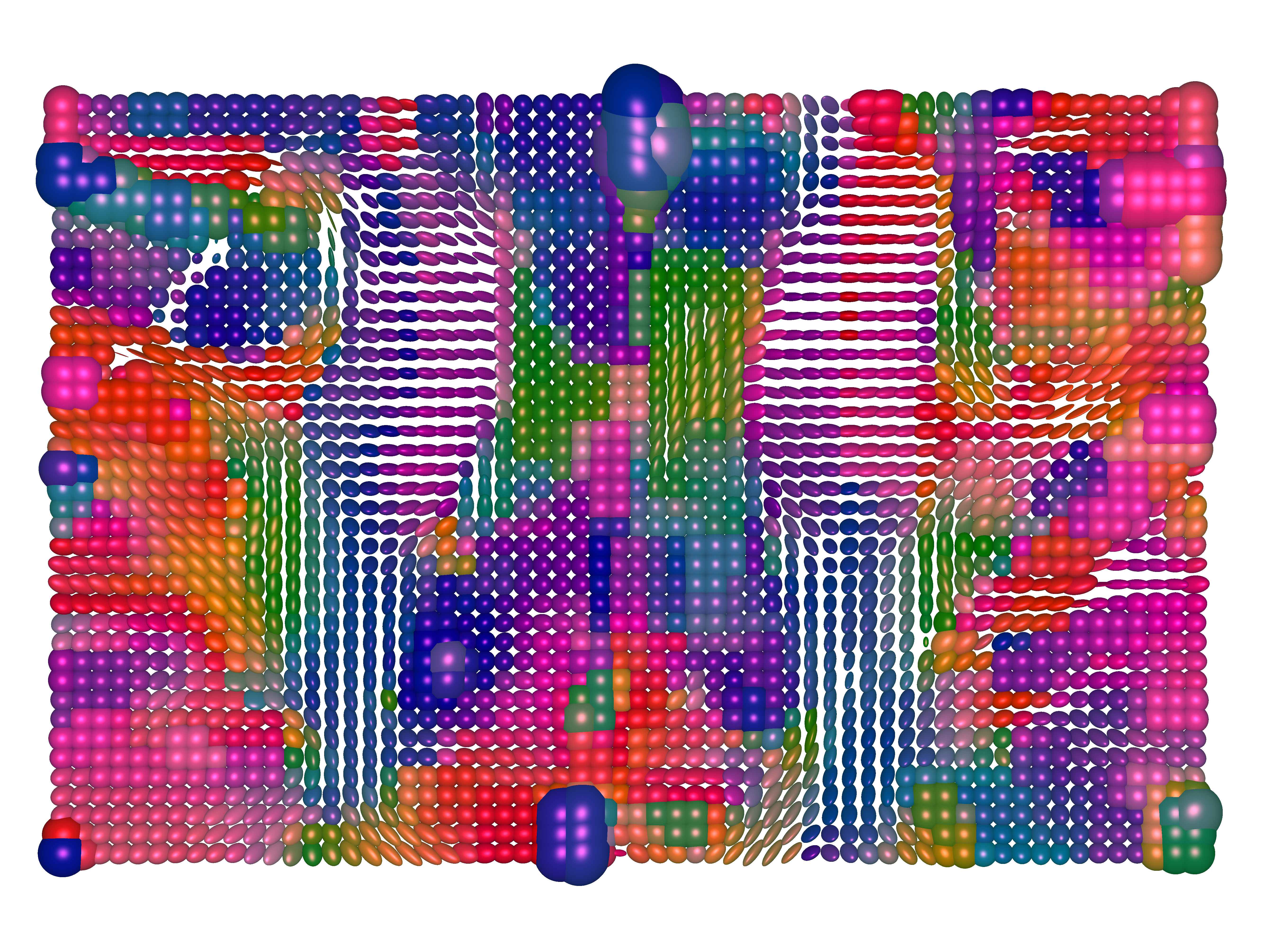}};
	\spy  on \nodeSpy in node [left] at \nodeWindow;
	\end{tikzpicture}
\end{subfigure}
\caption{\emph{Left:} Diffusion tensor image of a human brain (axial cut); \emph{Right:} Total variation denoising with $\ell^2$ data term using the cyclic proximal point algorithm (Algorithm \ref{alg:cyclic_splitting}) using $\alpha = \protect\input{\figfolder alpha.txt}.$ The runtime is $\protect\input{\figfolder runtime_BB_4000L2.txt}$ sec for $\protect\input{\figfolder NbIt.txt}$ iterations. 
The regularized image is much smoother than the original image. 
At the same time, strong changes of the orientations are preserved.}
\label{fig:dti_camino}
\end{figure}

\begin{figure}
\centering
\def\figfolder{./experiments/ex4/}
\def\figwidth{0.3\columnwidth}
\begin{subfigure}[t]{\figwidth}
	\includegraphics[width=\columnwidth, trim=65 4 65 10, clip]{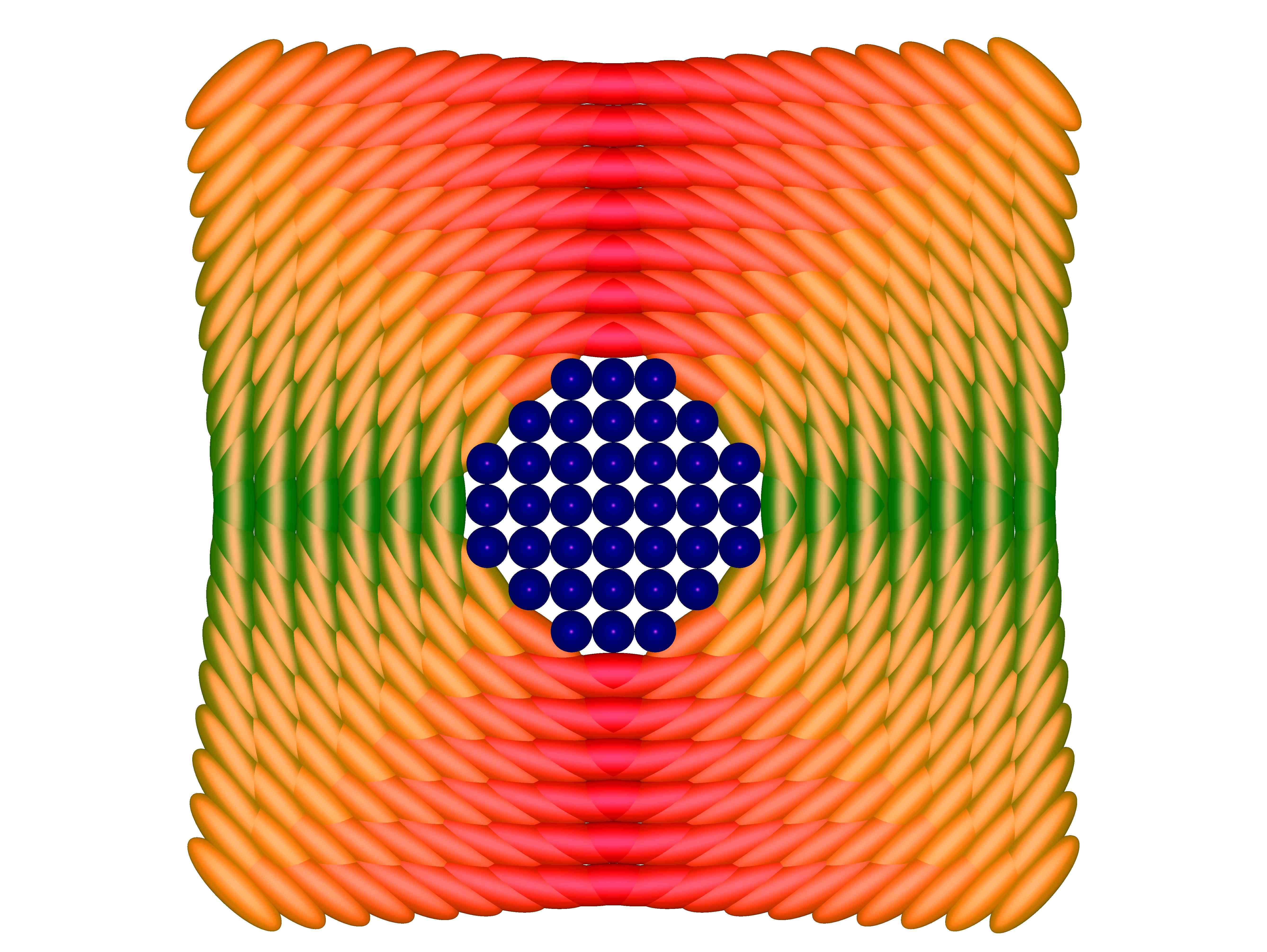}
	\subcaption{Synthetic DT image.}
\end{subfigure}
\hspace{0.05\textwidth}
\begin{subfigure}[t]{\figwidth}
	\includegraphics[width=\columnwidth, trim=65 4 65 10, clip]{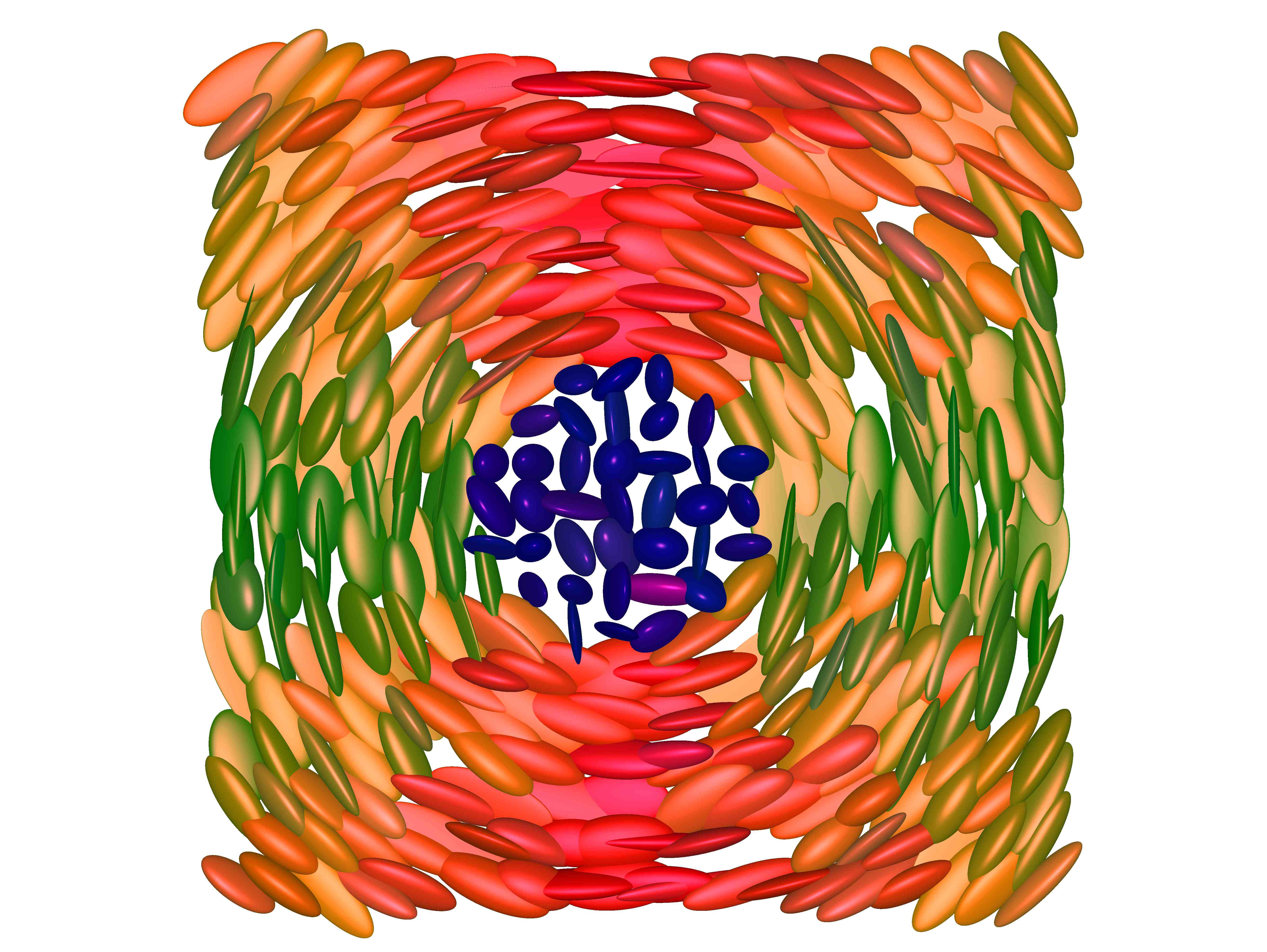}
		\subcaption{Rician noise with $\sigma = \protect\input{\figfolder noise_level.txt}.$}
\end{subfigure}
\\
\begin{subfigure}[t]{\figwidth}
	\includegraphics[width=\columnwidth,trim=65 4 65 10, clip]{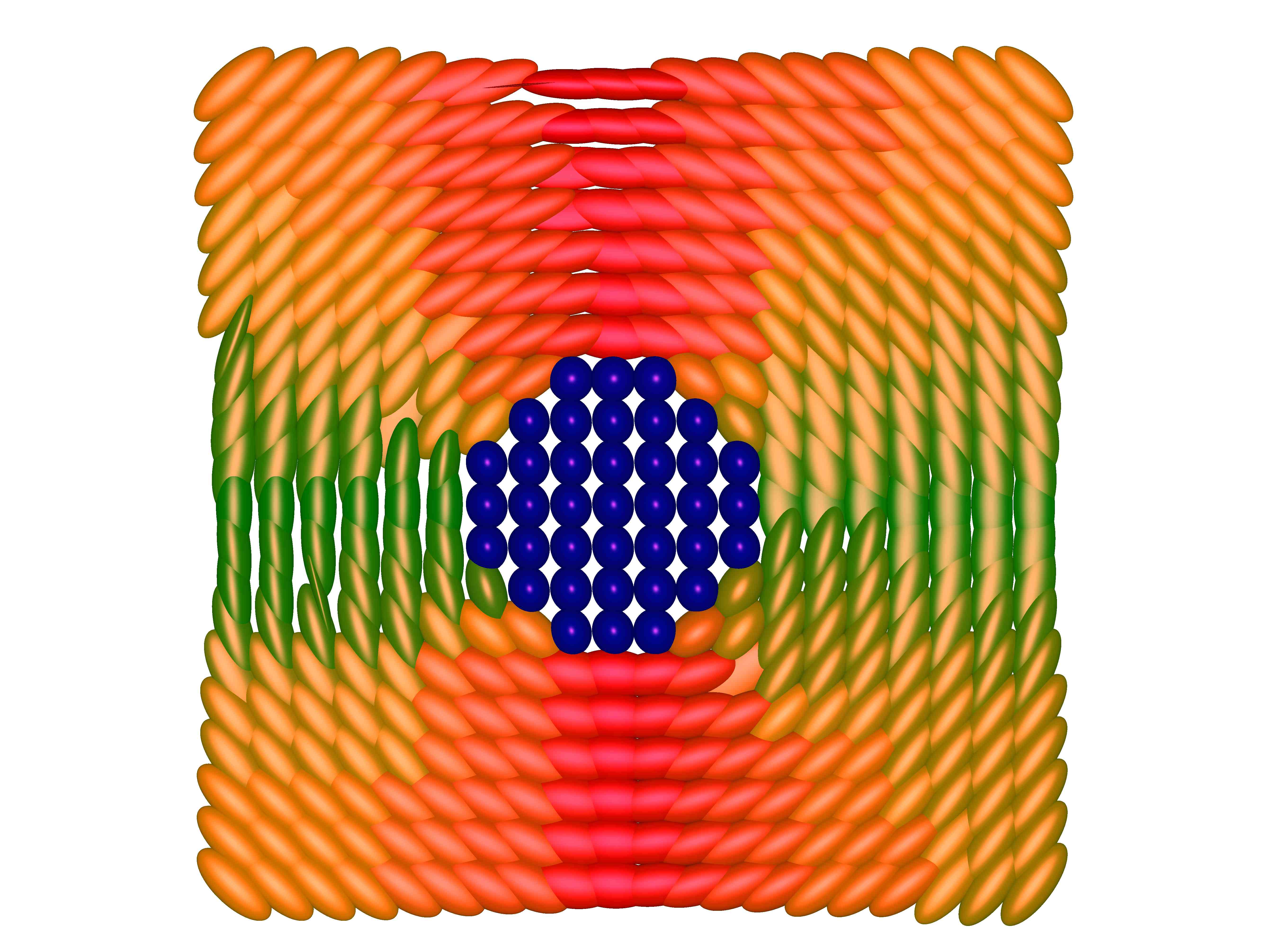}
	\subcaption{Algorithm~\ref{alg:cyclic_splitting}
	(72.2 sec).}
\end{subfigure} \hfill
\begin{subfigure}[t]{\figwidth}
	\includegraphics[width=\columnwidth, trim=65 4 65 10, clip]{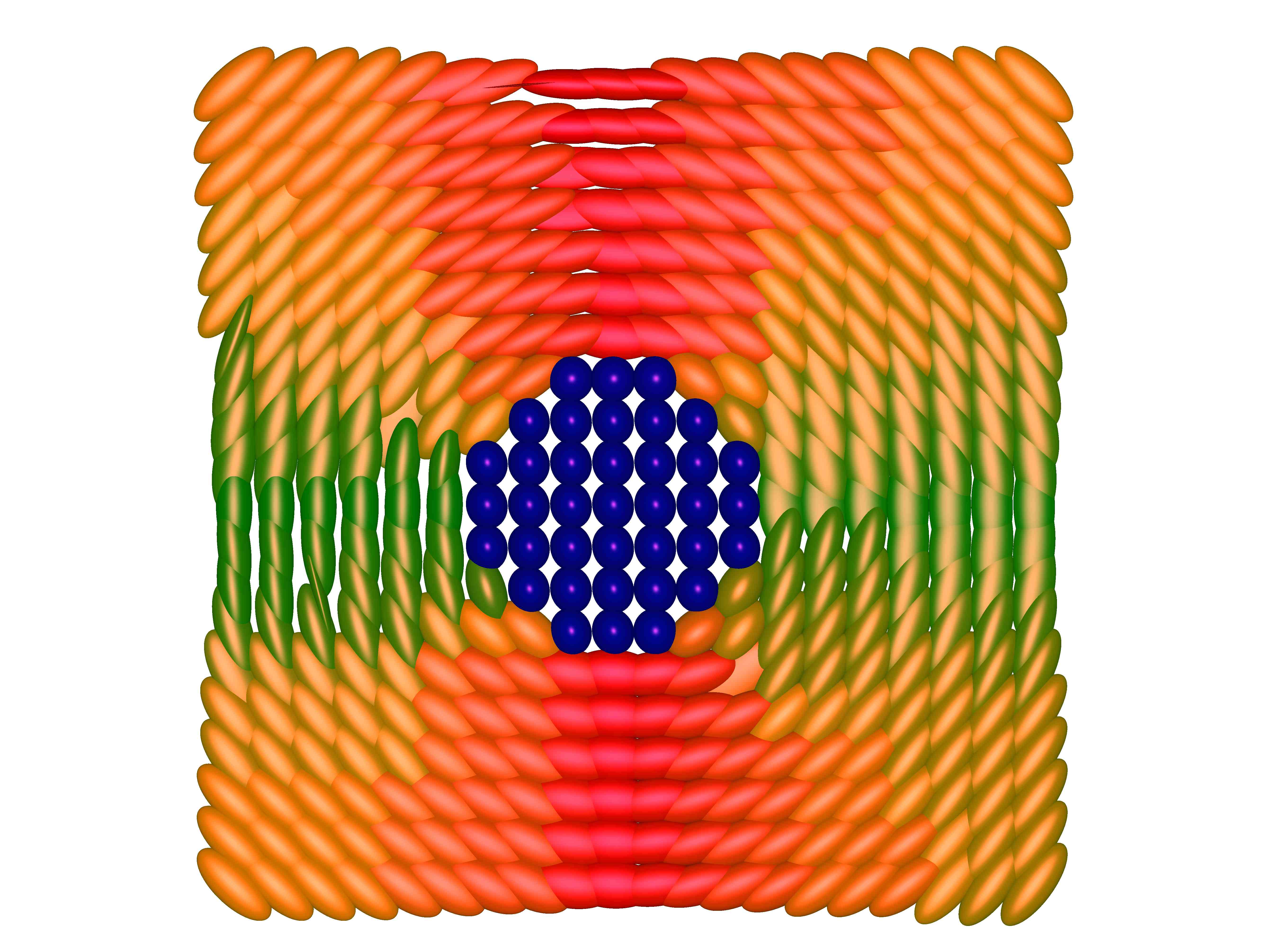}
	\subcaption{Algorithm~\ref{alg:parallel_splitting} ($\protect\input{\figfolder runtime_parallel.txt}$ sec).}
\end{subfigure}
\hfill
\begin{subfigure}[t]{\figwidth}
	\includegraphics[width=\columnwidth, trim=65 4 65 10, clip]{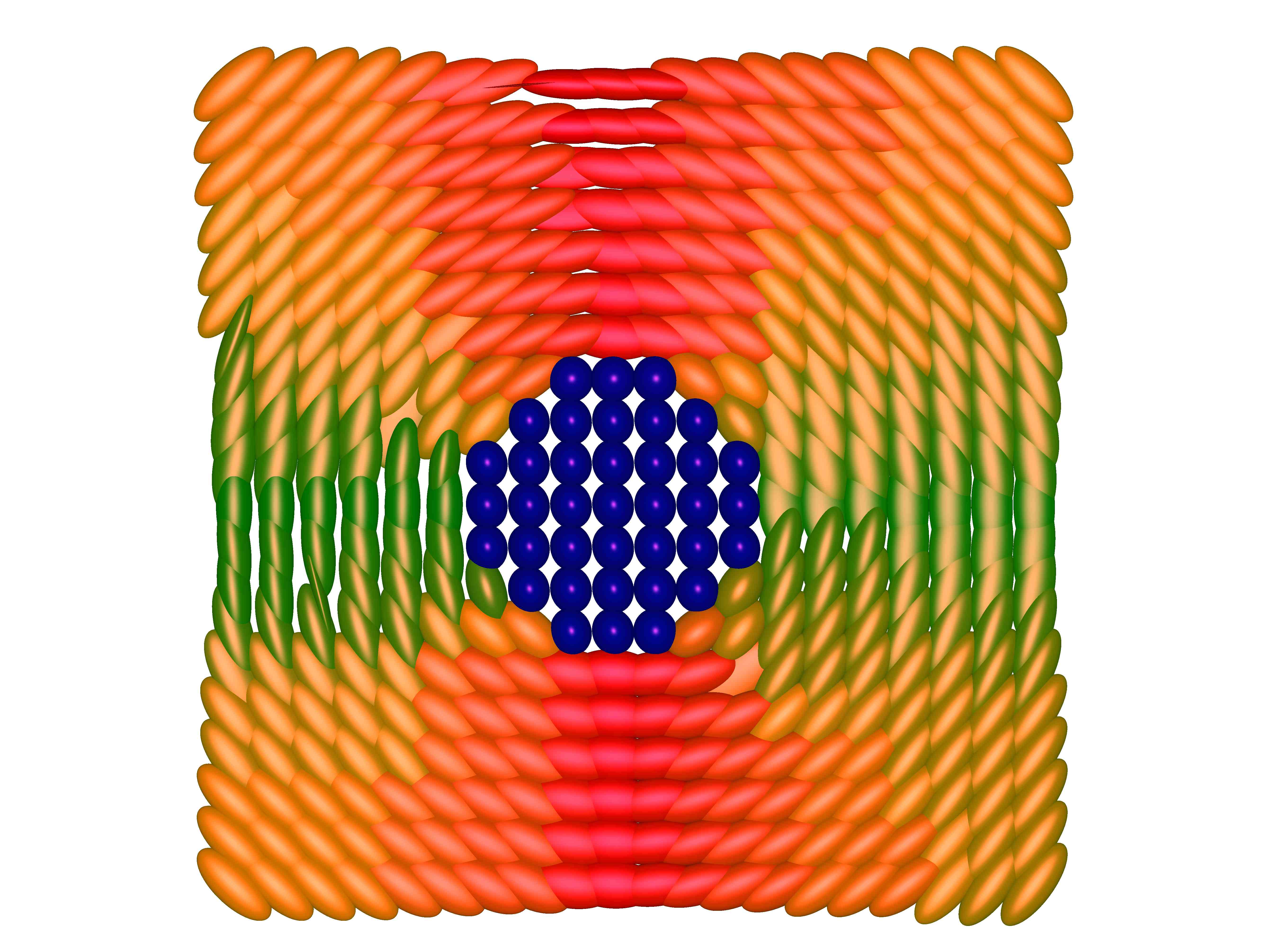}
	\subcaption{Fast variant of Algorithm~\ref{alg:parallel_splitting} ($\protect\input{\figfolder runtime_parallel_fast.txt}$ sec).}
\end{subfigure}
\figspace
\caption{
 $\ell^2$-TV regularization  of a diffusion tensor image with high noise level.
 Algorithm~\ref{alg:cyclic_splitting} as well as Algorithm~\ref{alg:parallel_splitting} and its fast variant converge to the same solution. 
 The TV regularization ($\alpha = \protect\input{\figfolder alpha.txt}$) removes almost all the noise and it preserves the sharp transitions. 
  The signal-to-noise-ratio improvement is $\deltaSNR = \protect\input{\figfolder delta_snr_bb.txt}$ in all three cases.
 The numbers in brackets denote the CPU time for $\protect\input{\figfolder NbIt.txt}$ iterations.
}
\label{fig:dti_synthetic}
\end{figure}

The data actually measured in DTI 
are so-called diffusion weighted images (DWI) $D_v(p)$ which capture the directional diffusivity in the direction $v$ at pixel $p.$
The relation between the diffusion tensor $S(p)$ and the DWIs $D_v(p)$ at the pixel $p$ is given by the Stejskal-Tanner equation
\begin{equation} \label{eq:StejTan}
   D_v(p) =  A_0 e^{- b \ v^T S(p) v} 
\end{equation}
with  constants $b,A_0>0.$ Typically $b = 800$ and $A_0 = 1000. $
Usually, $6$ to $30$ diffusion weighted images $D_v$ (with different directions $v$) are measured \cite[3. IV C]{johansen2009diffusion}.
Being magnetic resonance images the DWIs are corrupted by Rician noise
which arises from complex-valued Gaussian noise in the original frequency domain measurements \cite{basu2006rician}.
This means that assuming the model \eqref{eq:StejTan} the actual measurement in direction $v$ at pixel $p$ is given by 
\[
D'_v(p) = \sqrt{(X+ D_v(p))^2 + Y^2},
\]
with the Gaussian random variables $ X,Y \sim N(0,\sigma^2).$ 
Typically, the tensor $D_v(p)$ is obtained from the DWIs via a least square fit using the Stejskal-Tanner equation \eqref{eq:StejTan}.
In our synthetic examples, 
we impose Rician noise to $15$ diffusion weighted images $D'_v$ obtained from a synthetic diffusion tensor image $S$ by \eqref{eq:StejTan}.
Then we apply least square fitting to the noisy DWIs to obtain a noisy diffusion tensor image.

In our experiments we visualize the diffusion tensors by the isosurfaces of the corresponding quadratic forms. More precisely, the ellipse visualizing
the diffusion tensor $S(p)$ at pixel $p$ are the points $x$ fulfilling $(x-p)^T S(p)(x-p) = c,$ for some $c>0.$ 

In Figure~\ref{fig:dti_camino}, we apply $\ell^2$-TV minimization to real DTI data of a human brain. 
The data set stems from the Camino project \cite{cook2006camino} and is freely accessible.
We observe that TV minimization removes noise and preserves 
sharp boundaries between oriented structures.

In Figure~\ref{fig:dti_synthetic} we apply Algorithm~\ref{alg:cyclic_splitting} as well as Algorithm~\ref{alg:parallel_splitting} and its fast variant 
for $\ell^2$-TV minimization to a synthetic DTI on which we impose Rician noise.
We observe the denoising capabilities of the proposed algorithms under a relatively high noise level; minimization of the $\ell^2$-TV functional almost completely removes the noise
while preserving sharp boundaries at the same time.

\subsection{The one-dimensional sphere $\mathrm{S}^1$ -- InSAR images}

Synthetic aperture radar (SAR) is a radar technique for
sensing the earth's surface from measurements taken by aircrafts or satellites. 
Interferometric synthetic aperture radar (InSAR) images
consist of the phase difference between two SAR images,
recording a region of interest either from two different angles of view or at two different points in times.
Important  applications of InSAR are the creation of accurate digital elevation models and the detection of terrain changes; cf.~\cite{massonnet1998radar,rocca1997overview}.

\begin{figure}
\def\figfolder{./experiments/SAR/ex2/}
	\def\figwidth{0.24\columnwidth}
\begin{subfigure}[t]{\figwidth}
	\includegraphics[width= \textwidth]{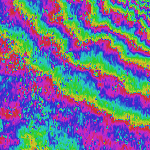}
	\caption{InSAR image (real data).}
	\end{subfigure}
	\hfill
	\begin{subfigure}[t]{\figwidth}
	\includegraphics[width= \textwidth]{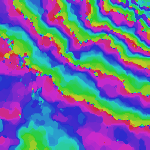}
	\caption{$\ell^2$-TV  $(\alpha=\protect\input{\figfolder alpha_2.txt}).$}
	\end{subfigure}
	\hfill
	\begin{subfigure}[t]{\figwidth}
	\includegraphics[width= \textwidth]{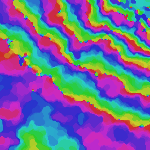}
	\caption{$\ell^1$-TV $(\alpha = \protect\input{\figfolder alpha_1.txt}).$}
		\end{subfigure}
	\hfill
	\begin{subfigure}[t]{\figwidth}
	\includegraphics[width= \textwidth]{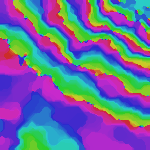}
	\caption{TV with Huber data term $(\alpha = \protect\input{\figfolder alpha_huber.txt}).$}
		\end{subfigure}
		\figspace
			\caption{
		Total variation denoising of an InSAR image of dimension $150 \times 150.$ The $\mathrm{S}^1$-valued data are visualized as hue component in the HSV color space. 
Total variation minimization reliably removes the noise while preserving the structure of the image. We observe that $\ell^1$ and Huber data terms are slightly more robust to outliers.  In all three cases the runtime is about 20 sec 
for $\protect\input{\figfolder NbIt.txt}$ iterations.
}		
\label{fig:SAR}
\end{figure}
As InSAR data consists of phase values, the natural data space of InSAR images is the one-dimensional sphere $S^1.$
The exponential mapping and its inverse have a particularly simple form
when regarding  $S^1$ as unit circle in the complex plane.
Then the exponential mapping is given by 
\[
  \exp_{a}(v) = e^{i (\theta + v)},
\]
where $a = e^{i \theta}$ and $v \in ]-\pi;\pi[$. 
For two non-antipodal points $a$ and $b$ the inverse exponential map reads 
\[
  \exp^{-1}_{a}(b) = \arg (b/a),
\]
which is the polar angle of the complex number $b/a.$
 The distance between two points on the sphere reads
$d(a,b) = |\arg (b/a)|.$

In Figure  \ref{fig:SAR}, we apply total variation denoising 
to a real InSAR image taken from \cite{rocca1997overview}. 
This experiment shows the different effects of total variation regularization of using different data terms. We use $\ell^2$ and $\ell^1$ terms as well as the Huber term (with the parameters 
$\tau = \sqrt{2}$ and $\omega = 1$ in the definition of the Huber function \eqref{eq:huber}). 
We used $600$ iterations of the cyclic proximal point algorithm.
We observe that TV regularization reduces the noise significantly.
The $\ell^1$ data term and the Huber data term appear to be more robust to outliers than the $\ell^2$ data term.

\subsection{$\R^2 \times S^1$-valued images --  Denoising in LCh color space}
\label{sec:LCh}

It was observed that total variation based denoising may give better results when using certain non-flat color models instead of the classical RGB color space \cite{chan2001total}.
One of these non-flat models is the HSV color space which leads to cylindrical data living in the product space $\R^2 \times S^1.$ 

We here use the LCh color space.
Similar to the HSV space it is a cylindrical space consisting of a luminance component $L \in \R^+_0,$ a chroma component $C \in \R^+_0,$ and a hue component $h \in S^1.$
The difference between  HSV and LCh is that the first derives directly from the RGB space,
whereas the latter derives from the Lab color space (also called $L^*a^*b^*$ space)
which is intended to better match the human visual perception than the technically motivated RGB space.
We perform the color space conversions using Matlab's built-in functions.
For the hue and range preserving enhancement of color images see \cite{nikolova2013fast}.

The exponential and the logarithmic mappings are given componentwise by the respective mappings on
 $\R^2$ and $S^1$. Note that in spite of this separability property, the proposed algorithm 
 is not equivalent to performing  the algorithm on $\R^2$ and $S^1$ separately (except for $p$=$q$=$2$). 
 The reason is that the  
 path length calculated according to Table~\ref{tab:waylength_reg} and Table~\ref{tab:waylength_data} except for $p,q$=$2$ 
 depend nonlinearly on the distance in the product manifold. 

In Figure~\ref{fig:LCh} 
we compare denoising in the RGB space with denoising in the LCh space.
The RGB example was computed using the split Bregman method for TV denoising which is a
state-of-the-art method for vectorial total variation regularization \cite{goldstein2009split, getreuer2012rudin}.
We optimized the corresponding model parameter with respect to the peak signal to noise ratio (PSNR)
 given by 
   \begin{equation*}
	\mathrm{PSNR}(x) = 10 \log_{10}\left(  \frac{ 3 m n \cdot (\max_{i,j,k}   |g_{i,j,k}|)^2}{  \sum_{i,j,k} | g_{i,j,k} - x_{i,j,k} |^2} \right).
\end{equation*}
where $g \in \R^{n \times m \times 3}$ denotes the ground truth (in RGB space).
In Figure~\ref{fig:LCh} we observe that the LCh color space denoising can indeed lead to better results 
than the vectorial RGB denoising.

\begin{figure}
\def\figfolder{./experiments/ex16/}
	\def\figwidth{0.24\columnwidth}
\begin{subfigure}[t]{\figwidth}
	\includegraphics[width= \textwidth]{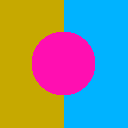}
	\caption{Original image.}
	\end{subfigure}
		\hfill
	\begin{subfigure}[t]{\figwidth}
	\includegraphics[width= \textwidth]{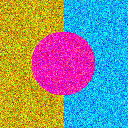}
	\caption{Gaussian noise (PSNR: 15.64).}
	\end{subfigure}
	\hfill
	\begin{subfigure}[t]{\figwidth}
		\includegraphics[width= \textwidth]{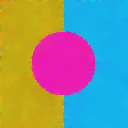}
 	\caption{$\ell^2$-TV in RGB space ($\alpha = \protect\input{\figfolder lambda_opt.txt}$, PSNR: $\protect\input{\figfolder max_psnr_split_bregman_l2.txt}$).}
	\end{subfigure}
		\hfill
	\begin{subfigure}[t]{\figwidth}
	\includegraphics[width= \textwidth]{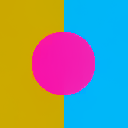}
	\caption{$\ell^2$-TV in LCh space ($\alpha = \protect\input{\figfolder alpha_l2.txt},$ PSNR: $\protect\input{\figfolder max_psnr_bb_l2.txt}$).}
	\end{subfigure}
	\figspace
	\caption{
	TV denoising in different color spaces.
	We see that measuring the distance in the non-flat LCh metric can lead to higher reconstruction quality 
	 than in the RGB color space. The runtime is 13 sec for $\protect\input{\figfolder NbIt.txt}$ iterations.
	}		
	\label{fig:LCh}
\end{figure}

%%%%%%%%%
\subsection{$S^2$-valued images}
%%%%%%%%%
We next apply our methods to images taking values in the two-dimensional sphere $S^2.$ %which is the unit sphere in $\R^3.$
For example, spherical data appear in image processing in the context of chromaticity-based color models 
\cite{chan2001total,vese2002numerical} and as orientation fields of three dimensional images \cite{rezakhaniha2012experimental}.

For a unit vector $a$ on the unit sphere $S^2$ in $\R^3$ and a non-zero tangent vector $v$ to the sphere at the point $a,$ the exponential mapping is given by 
\[
  \exp_{a}(v) =  a \cdot \cos \|v\|  + \frac{a \cdot\sin \|v\|}{\|v\|}.
\]
The inverse $\exp^{-1}_{a}$ of the exponential mapping is well defined for non-antipodal points $a$ and $b$ and given by 
\[
  \exp^{-1}_{a}(b) =  \arccos(\ps{a,b}) \cdot \frac{b-\ps{b,a}a}{\|b-\ps{b,a}a\|},
\]
where $\ps{\cdot,\cdot}$ denotes the standard inner product in $\R^3.$ 
The distance between $a$ and $b$ is 
$d(a,b) = \arccos(\ps{a,b}).$

We test the denoising potential of our algorithm on a (synthetic) spherical-valued image.
 In the context of directional statistics a popular noise model on $S^2$ uses the von Mises-Fisher distribution having the probability density 
  \[
    f(x) = c(\kappa) \exp (\kappa \ps{x,\mu}).
  \]
Here the parameter $\kappa >0$ expresses the concentration around the mean orientation $\mu \in S^2$ -- the higher $\kappa$, the more concentrated the distribution. 
The constant $c(\kappa)$ is used for normalization to obtain a probability measure. For each data point $x_{ij} \in S^2$, we consider the above distribution with $\mu = x_{ij}$ and draw a sample. 
For the simulation of the distribution we used the implementation \cite{jung2010randomgenerator}; see \cite{wood1994simulation} for a description of the algorithm. 
In Figure~\ref{fig:s2_exp}, we observe that the noise is almost completely removed by TV minimization and that the edges are retained.
\begin{figure}
\def\figfolder{./experiments/S2/ex7_S2/}
	\def\figwidth{0.32\columnwidth}
	\includegraphics[width=\figwidth]{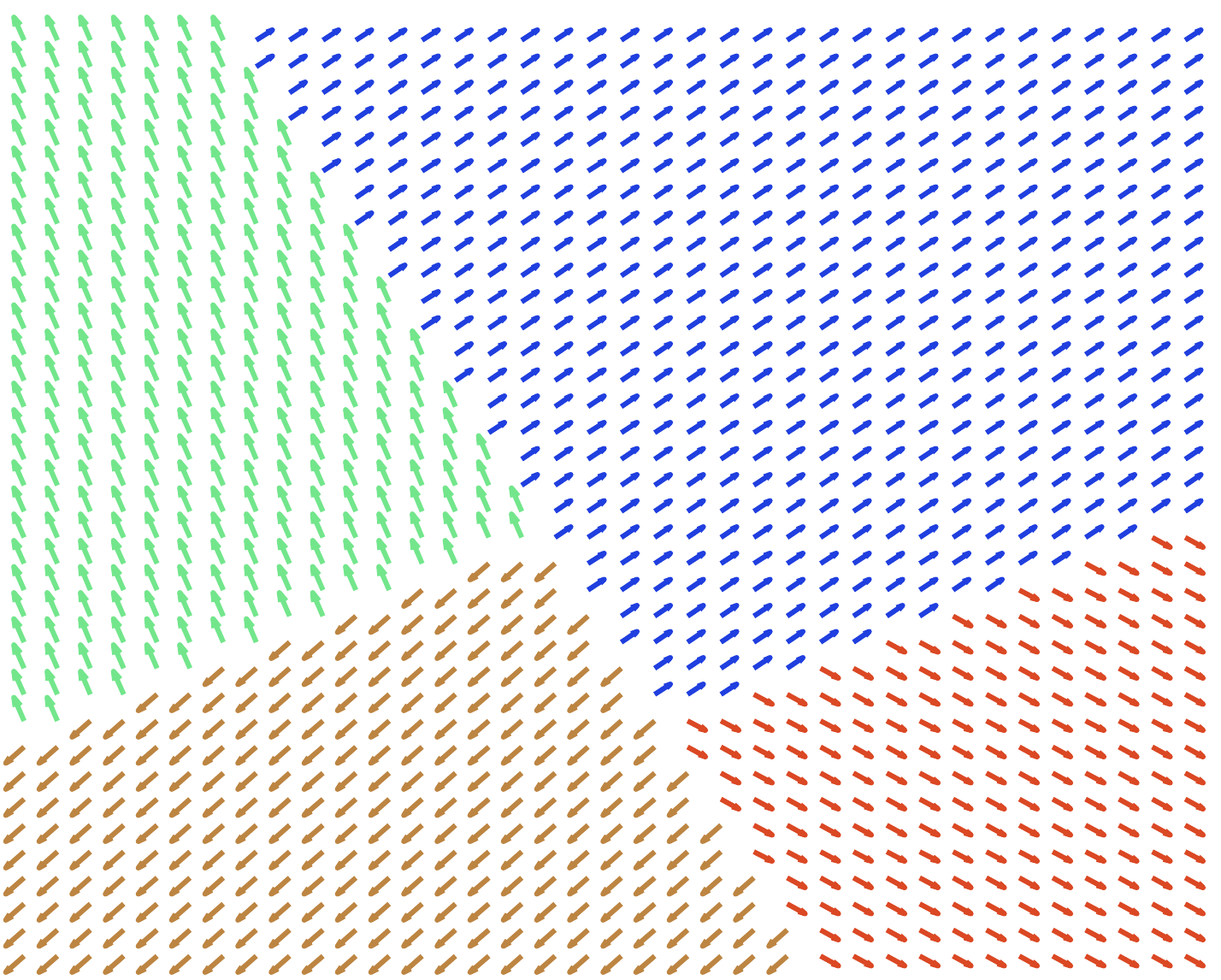}
	\hfill
		\includegraphics[width=\figwidth]{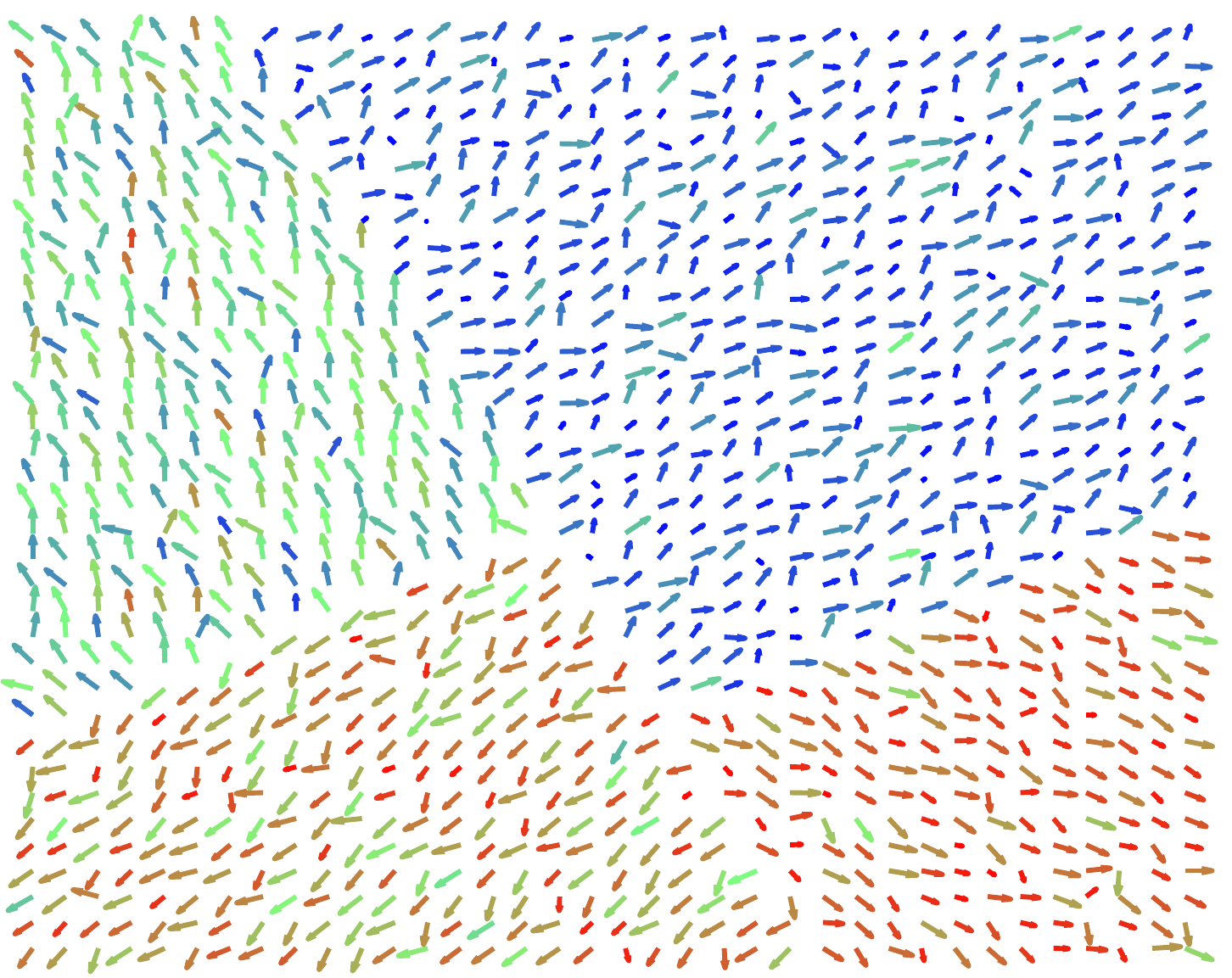}
		\hfill
\includegraphics[width=\figwidth]{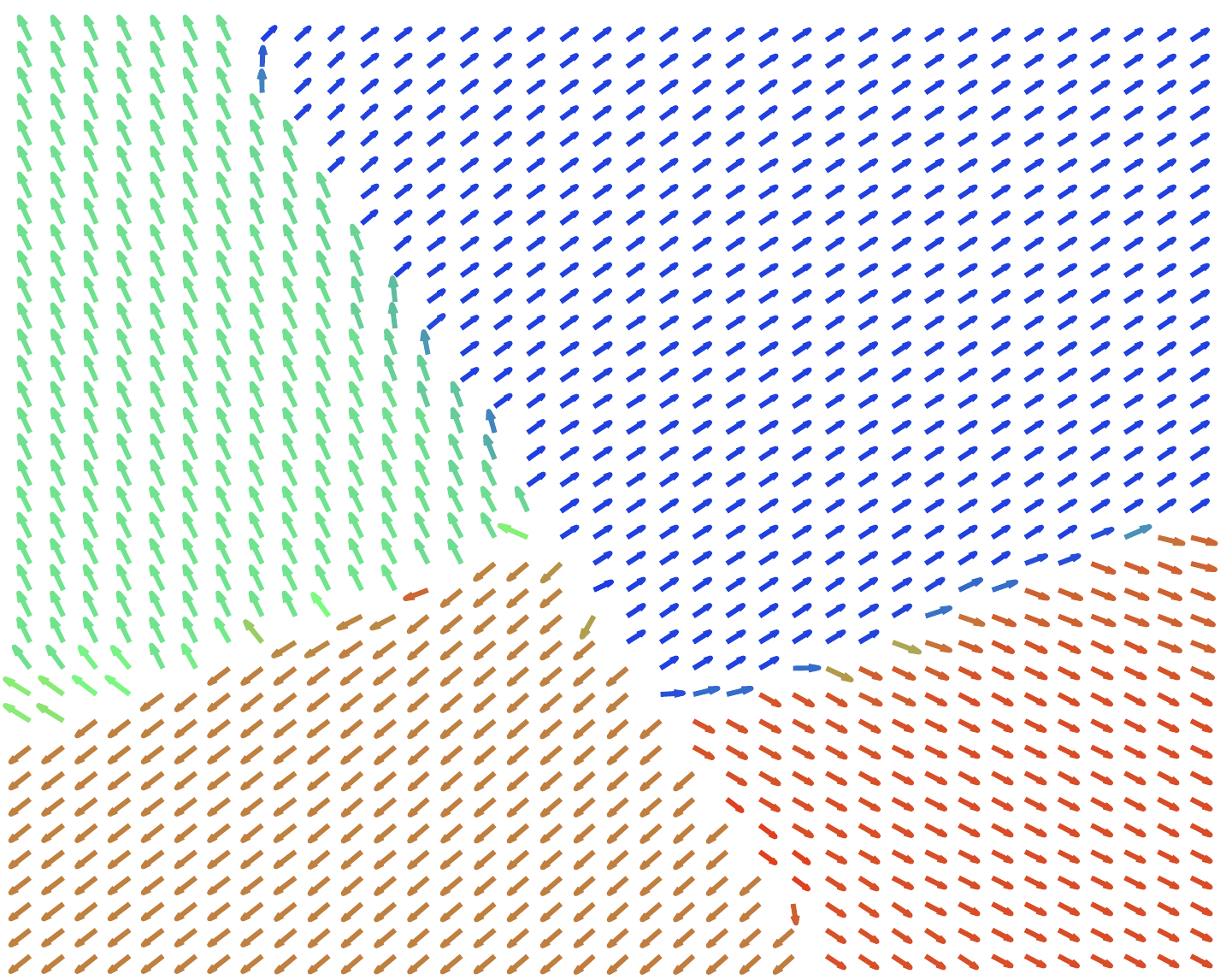}
		\caption{
		Denoising of an $S^2$-valued image. The polar angle is coded both as length of the vectors and as color (red pointing towards the reader, blue away from the reader).
		\emph{Left:} Original;
\emph{Center:} Von Mises-Fisher noise of level $\kappa =\protect\input{\figfolder kappa.txt};$ 
\emph{Right:} $\ell^{\protect\input{\figfolder p.txt}}$-TV regularization
using $\alpha = \protect\input{\figfolder alpha.txt}.$ The noise is almost completely removed whereas the jumps are preserved ($\deltaSNR = \protect\input{\figfolder delta_snr_bb.txt}$).
The runtime is $\protect\input{\figfolder runtime_BB.txt}$ sec for $\protect\input{\figfolder NbIt.txt}$ iterations.
}	
\label{fig:s2_exp}			
\end{figure}

%%%%%%%%%
\subsection{SO(3) data}
%%%%%%%%%

Measurements which involve the orientations of rigid objects in three-dimensional space lead to data which take their values in the rotation group $\mathrm{SO}(3).$
Examples of $\mathrm{SO}(3)$-valued data are aircraft orientations \cite{rahman2005multiscale} and  protein alignments \cite{green2006bayesian}. 
They also appear in the context of tracking 3D rotational data arising in robotics \cite{drummond2002real}; see also \cite{moakher2002means} for connections with directional statistics.  
 
The special orthogonal group $\mathrm{SO}(3)$ consists of all orthogonal $3\times 3$ matrices with determinant one,
i.e. 
\[
  SO(3) = \left\{ Q \in GL(3) \;:\; Q^t Q = I_3, \det Q = 1  \right\}.
\] 
As usual for matrix groups, we only consider the tangent space to $\mathrm{SO}(3)$ in the identity matrix $I_3.$ It is given by the space of $3 \times 3$ skew-symmetric matrices $\mathrm{so}(3)$. 
Identifying the tangent space at an arbitrary point $P$ with the tangent space at $I_3$ (via the differential of the left group action) the exponential mapping  
$\exp_{P}: \mathrm{so}(3) \to \mathrm{SO}(3)$ in the point $P$ is given by 
\[
  \exp_{P}(W) = \exp(W)P.
\]
Here $\exp$ denotes the matrix exponential. For $P,Q \in SO(3)$, 
the ``inverse'' of the above exponential mapping reads
\[
  \exp^{-1}_{P}(Q) =  \log(QP^t),
\]
where $\log$ denotes the principal logarithm (which may be viewed as componentwise principal logarithm on the eigenvalues). The distance between $P$ and $Q$ equals
the Frobenius norm of $\log(QP^t).$ 

For the matrix operations needed above there are closed form expression available; see e.g. \cite{moakher2002means}.
More precisely, to compute the matrix exponential of a  skew-symmetric matrix $W$ we use the Rodrigues formula
\[
\exp(W) = 
  I_3 + \frac{\sin(a)}{a} W + \frac{1-\cos a}{a^2}  W^2,
  \quad \mbox{ for } a = \sqrt{\Tr(W^t W)} > 0   .
\]
For $a = 0,$ we have $\exp(W)=I_3.$ 
Concerning the principal matrix logarithm of a rotation matrix $P$, 
we let $\theta= \arccos (\Tr (P) -1)/2.$ If $\theta = 0,$ then $\log(P) = 0,$ the zero matrix.
For $|\theta|< \pi,$ the principal logarithm of $P$ is given by
\[
\log(P) = \frac{\theta \cdot Y}{2 \sin(\theta)}, \qquad \mbox{ where } Y = (P-P^t)/2.
\]

\begin{figure}
	\centering
	\def\figfolder{./experiments/SO3/ex6_SO3/}
	\def\figwidth{0.45\columnwidth}
	\begin{subfigure}[t]{\figwidth}
	\includegraphics[width=\textwidth]{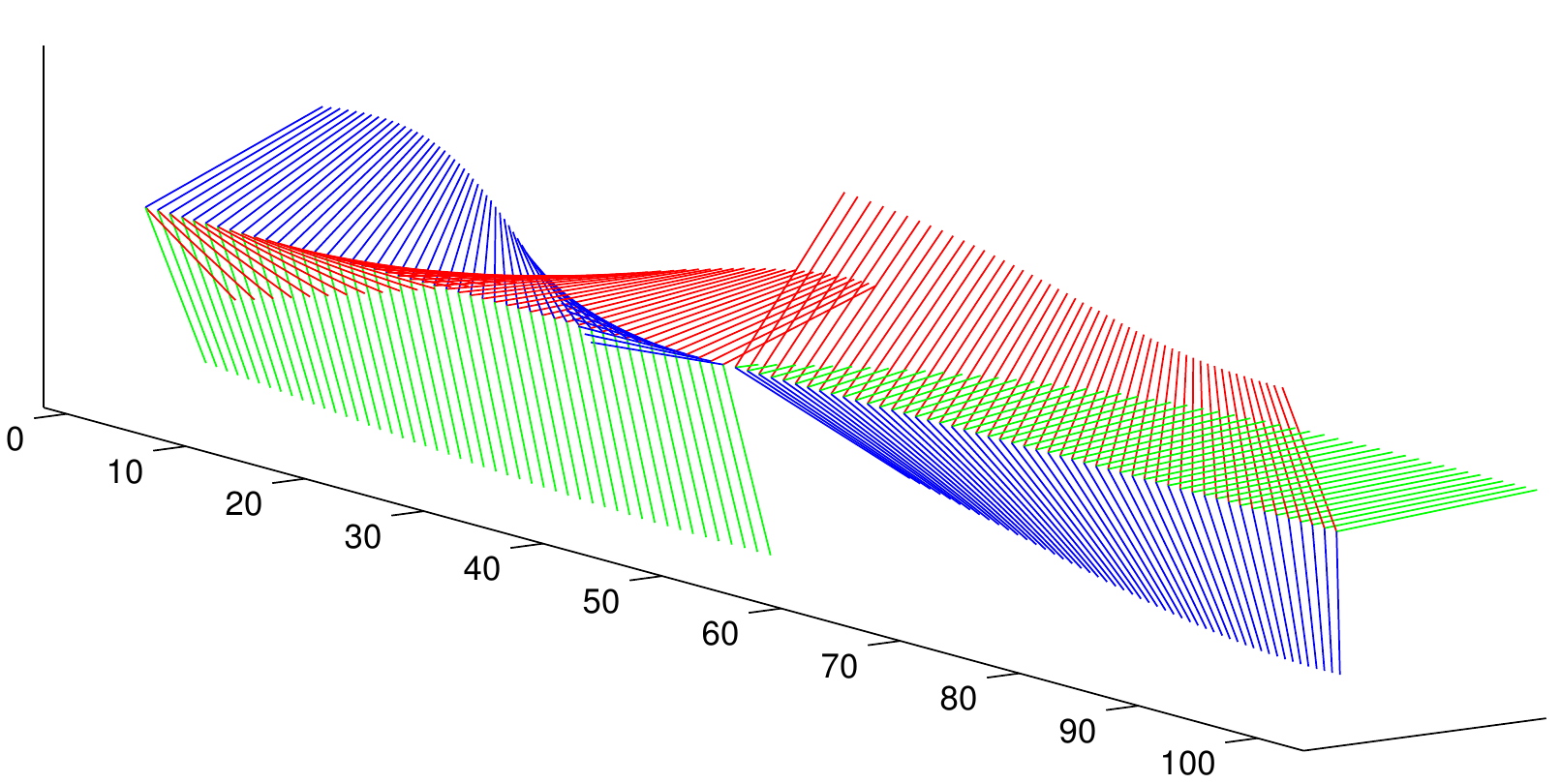}
	\caption{Original.}
		\end{subfigure}
		\begin{subfigure}[t]{\figwidth}
	\includegraphics[width=\textwidth]{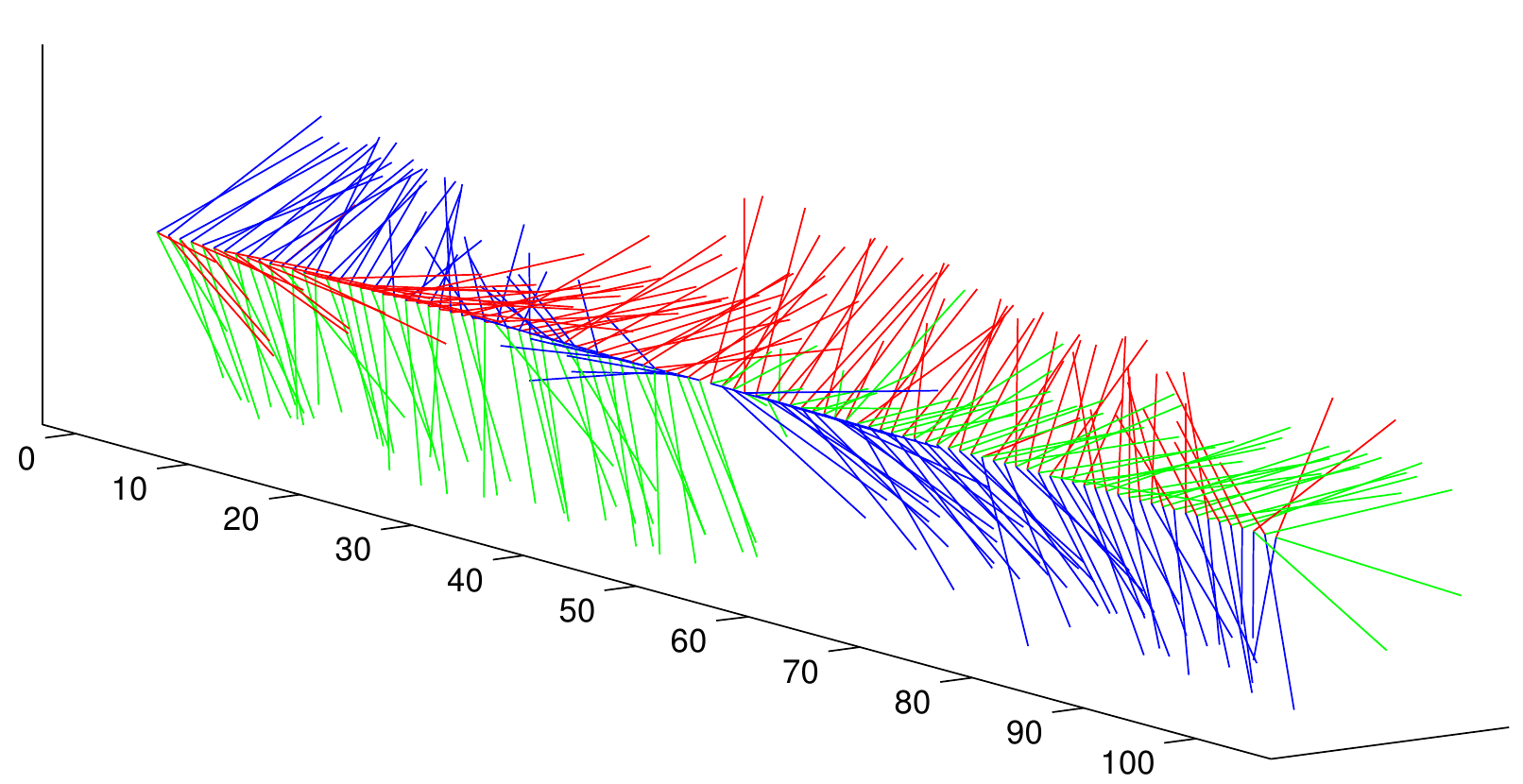}
	\caption{Fisher noise of level $\kappa = 75.$}
		\end{subfigure}
	\\
		\begin{subfigure}[t]{\figwidth}
	\includegraphics[width=\textwidth]{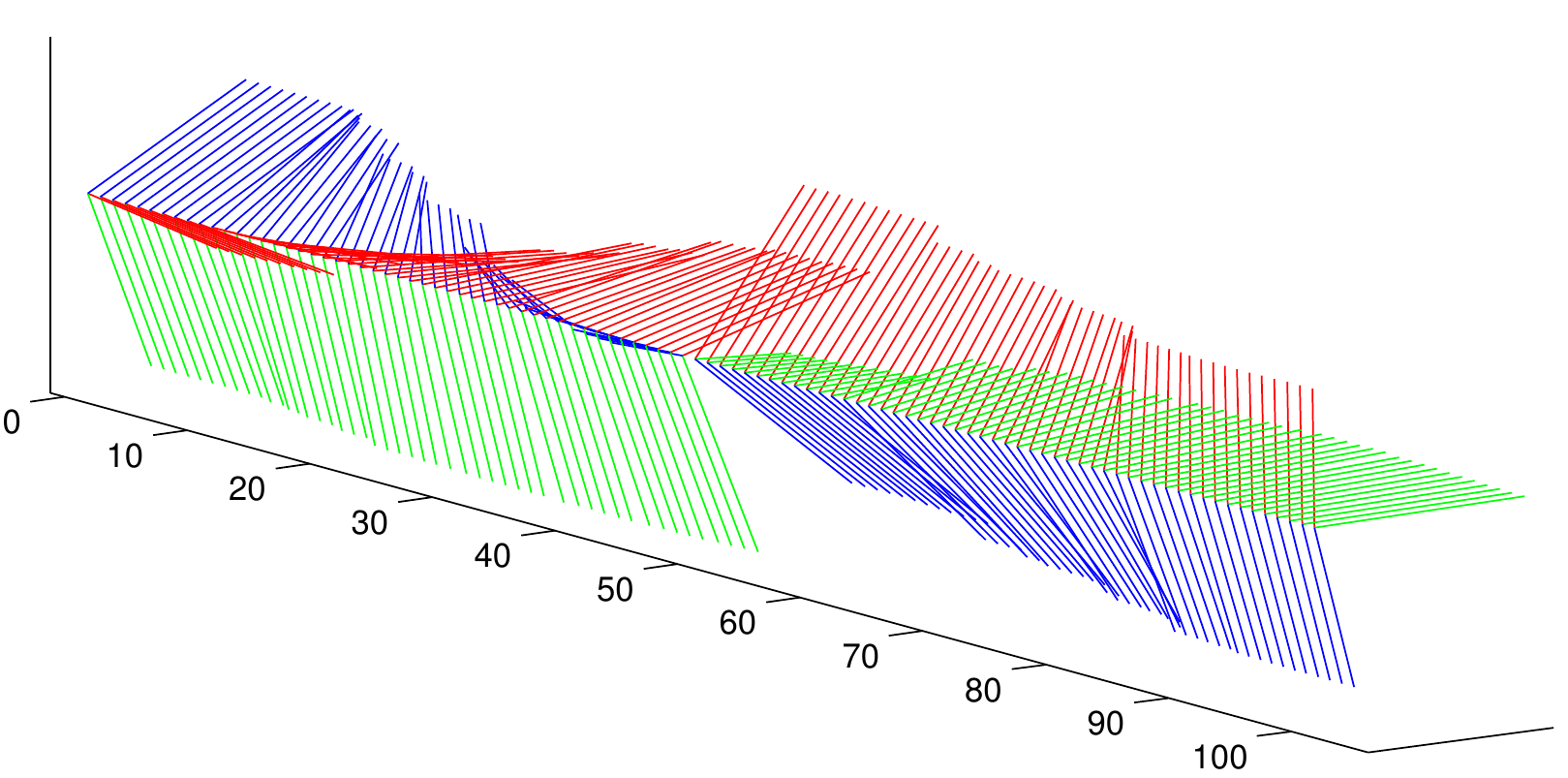}
	\caption{$\ell^2$-TV 
($\alpha = \protect\input{\figfolder alpha_TV.txt},$
 $\deltaSNR = \protect\input{\figfolder delta_snr.txt}$).
}
\end{subfigure}
		\begin{subfigure}[t]{\figwidth}
	\includegraphics[width=\textwidth]{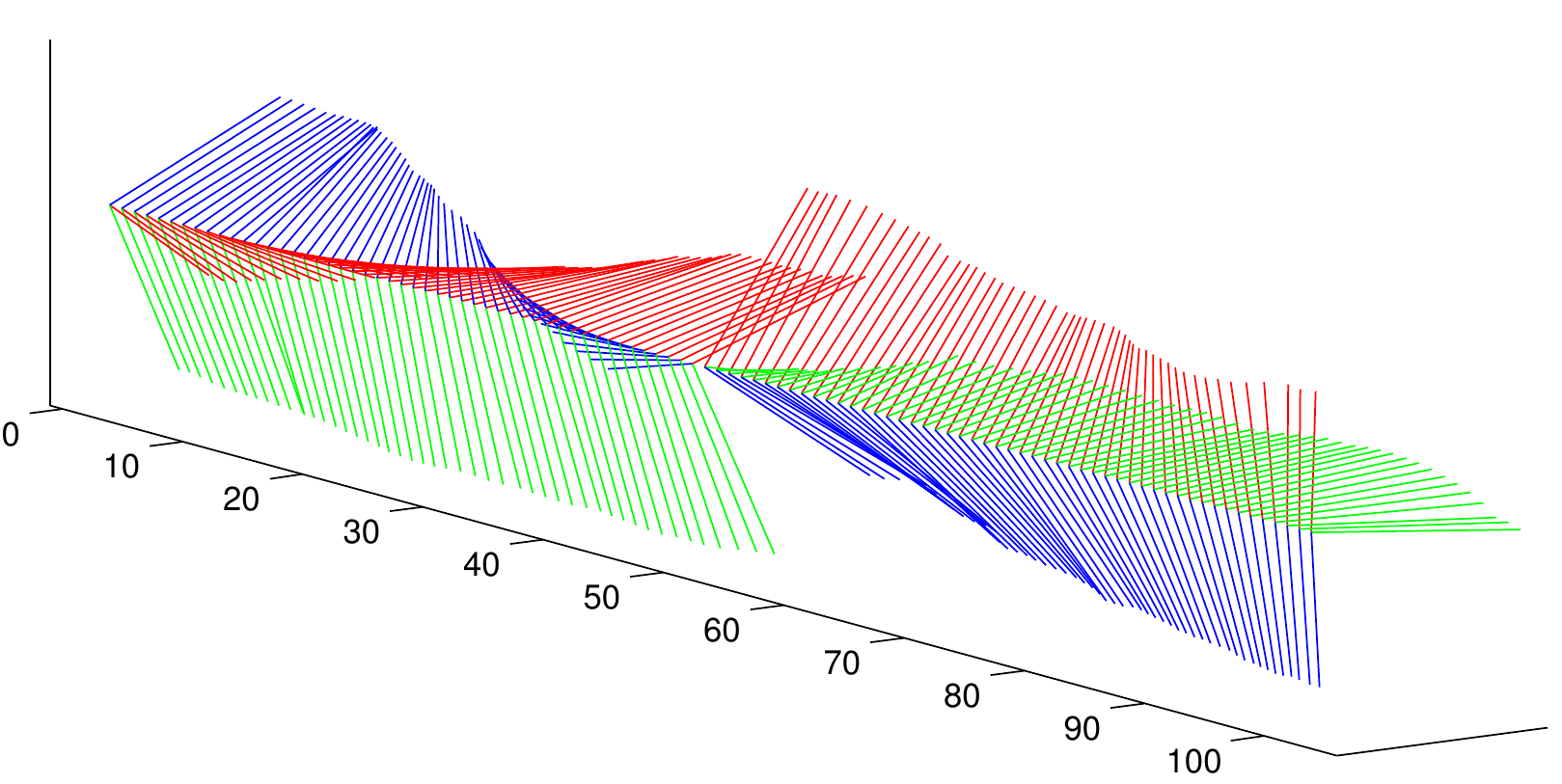}
	\caption{$\ell^2$-Huber 
 ($\alpha = \protect\input{\figfolder alpha_Huber.txt},$  
 $\deltaSNR = \protect\input{\figfolder delta_snr_huber.txt}$).}
			\end{subfigure}
			\figspace
\caption{Denoising of a $\mathrm{SO}(3)$-valued time-series;
 TV regularization removes the noise and  preserves the jump. 
 The Huber regularization term gives even better results with less staircasing effects. The runtimes for TV and Huber regularization are 
 $\protect\input{\figfolder runtime.txt}$ sec and $\protect\input{\figfolder runtime_huber.txt}$ sec, respectively, for $\protect\input{\figfolder NbIt.txt}$ iterations.
}
\label{fig:so3}
\end{figure}

In Figure~\ref{fig:so3}, we consider a synthetic $1D$ signal consisting of $130$ rotation matrices (visualized as tripods). 
The signal varies smoothly except for a single jump at the $50^{th}$ matrix. 
We simulate noisy data using the matrix Fisher distribution  \cite{khatri1977mises}
which is given by the density 
\[
  f(X) = c_F \exp \left\{ \Tr (F^t X)\right\}, \; X \in SO(3).
\]
The matrix $F$ is a fixed $3 \times 3$ parameter matrix which describes the mode and the concentration of the distribution. 
In the isotropic case, the concentration of the distribution can be described by a single parameter $\kappa > 0$ which can be regarded as noise level. 
A small value of $\kappa$ corresponds to a high level of noise. 
Our simulation uses a method recently introduced in  \cite{ganeiber2012estimation} and relies on the sampling of quaternions following a related  Bingham distribution. To simulate the latter we used the code from \cite{brubaker2012family} which implements the method in \cite{hoff2009simulation}. For details we refer to \cite[Chapter 5]{ganeiber2012estimation}. 

The results in Figure~\ref{fig:so3} show that the proposed algorithm removes the noise. The resulting signal is smoothed 
while the jump is preserved.
In that experiment, we also compare total variation with Huber regularization terms. 
We see that the Huber regularization exhibits less staircasing artifacts than TV regularization.

\section{Conclusion and future research}
In this article we have developed proximal point algorithms for total variation minimization for manifold-valued data.   
Our experiments show the denoising capability of the developed algorithms in various manifolds appearing in applications. 
For Hadamard spaces, we obtain convergence towards a global minimizer of the TV functional.

In future work, we address Blake-Zisserman and Potts functionals for manifold-valued data. 

\section*{Acknowledgement}
This work was supported by the German Federal Ministry for Education and Research under SysTec Grant 0315508 and by the
European Research Council (ERC) under the European Union's Seventh Framework Programme (FP7/2007-2013) / ERC grant agreement no.~267439.
The authors would like to thank Gabriele Steidl for valuable discussions and for pointing out the PPXA algorithm in \cite{combettes2011proximal} 
which led us to consider the parallel proximal algorithms in this paper.
They would also like to thank Klaus Hahn for valuable discussions.

\bibliographystyle{plain}
\bibliography{tvForManifoldBib}

\begin{thebibliography}{10}

\bibitem{alliney1992digital}
S.~Alliney.
\newblock Digital filters as absolute norm regularizers.
\newblock {\em IEEE Transactions on Signal Processing}, 40:1548--1562, 1992.

\bibitem{anagaw2012edge}
A.~Anagaw and M.~Sacchi.
\newblock Edge-preserving seismic imaging using the total variation method.
\newblock {\em Journal of Geophysics and Engineering}, 9:138, 2012.

\bibitem{azagra2005proximal}
D.~Azagra and J.~Ferrera.
\newblock Proximal calculus on {R}iemannian manifolds.
\newblock {\em Mediterranean Journal of Mathematics}, 2:437--450, 2005.

\bibitem{bavcak2013computing}
M.~Ba{\v{c}}{\'a}k.
\newblock Computing medians and means in {H}adamard spaces.
\newblock Preprint, 2013.

\bibitem{ballmann1985manifolds}
W.~Ballmann, M.~Gromov, and V.~Schroeder.
\newblock {\em Manifolds of nonpositive curvature}.
\newblock Birkh{\"a}user, Boston, 1985.

\bibitem{basser1994mr}
P.~Basser, J.~Mattiello, and D.~LeBihan.
\newblock {MR} diffusion tensor spectroscopy and imaging.
\newblock {\em Biophysical journal}, 66:259--267, 1994.

\bibitem{basu2006rician}
S.~Basu, T.~Fletcher, and R.~Whitaker.
\newblock Rician noise removal in diffusion tensor {MRI}.
\newblock In {\em Medical Image Computing and Computer-Assisted Intervention
  2006}, pages 117--125. Springer, 2006.

\bibitem{Bertsekas2011in}
D.~Bertsekas.
\newblock Incremental proximal methods for large scale convex optimization.
\newblock {\em Mathematical Programming}, 129:163--195, 2011.

\bibitem{boyd2011distributed}
S.~Boyd, N.~Parikh, E.~Chu, B.~Peleato, and J.~Eckstein.
\newblock Distributed optimization and statistical learning via the alternating
  direction method of multipliers.
\newblock {\em Foundations and Trends in Machine Learning}, 3:1--122, 2011.

\bibitem{bridson1999metric}
M.~Bridson and A.~Haefliger.
\newblock {\em Metric spaces of non-positive curvature}.
\newblock Springer, Berlin, 1999.

\bibitem{brubaker2012family}
M.~Brubaker, M.~Salzmann, and R.~Urtasun.
\newblock A family of {MCMC} methods on implicitly defined manifolds.
\newblock In {\em International Conference on Artificial Intelligence and
  Statistics}, pages 161--172, 2012.

\bibitem{chambolle2004algorithm}
A.~Chambolle.
\newblock An algorithm for total variation minimization and applications.
\newblock {\em Journal of Mathematical Imaging and Vision}, 20:89--97, 2004.

\bibitem{chambolle1997image}
A.~Chambolle and P.-L. Lions.
\newblock Image recovery via total variation minimization and related problems.
\newblock {\em Numerische Mathematik}, 76:167--188, 1997.

\bibitem{chambolle2011first}
A.~Chambolle and T.~Pock.
\newblock A first-order primal-dual algorithm for convex problems with
  applications to imaging.
\newblock {\em Journal of Mathematical Imaging and Vision}, 40:120--145, 2011.

\bibitem{chan2005aspects}
T.~Chan and S.~Esedoglu.
\newblock Aspects of total variation regularized {$L^1$} function
  approximation.
\newblock {\em SIAM Journal on Applied Mathematics}, 65:1817--1837, 2005.

\bibitem{chan2001total}
T.~Chan, S.~Kang, and J.~Shen.
\newblock Total variation denoising and enhancement of color images based on
  the {CB} and {HSV} color models.
\newblock {\em Journal of Visual Communication and Image Representation},
  12:422--435, 2001.

\bibitem{Chaux2007Var}
C.~Chaux, P.~Combettes, J.-C. Pesquet, and V.~Wajs.
\newblock A variational formulation for frame based inverse problems.
\newblock {\em Inverse Problems}, 23:1495--1518, 2007.

\bibitem{chefd2004regularizing}
C.~Chefd'Hotel, D.~Tschumperl{\'e}, R.~Deriche, and O.~Faugeras.
\newblock Regularizing flows for constrained matrix-valued images.
\newblock {\em Journal of Mathematical Imaging and Vision}, 20:147--162, 2004.

\bibitem{chen2005noise}
B.~Chen and E.~Hsu.
\newblock Noise removal in magnetic resonance diffusion tensor imaging.
\newblock {\em Magnetic Resonance in Medicine}, 54:393--401, 2005.

\bibitem{chen2006total}
T.~Chen, W.~Yin, X.~Zhou, D.~Comaniciu, and T.~Huang.
\newblock Total variation models for variable lighting face recognition.
\newblock {\em IEEE Transactions on Pattern Analysis and Machine Intelligence},
  28:1519--1524, 2006.

\bibitem{clason2010semismooth}
C.~Clason, B.~Jin, and K.~Kunisch.
\newblock A semismooth {N}ewton method for {$L^1$} data fitting with automatic
  choice of regularization parameters and noise calibration.
\newblock {\em SIAM Journal on Imaging Sciences}, 3:199--231, 2010.

\bibitem{combettes2011proximal}
P.~Combettes and J.-C. Pesquet.
\newblock Proximal splitting methods in signal processing.
\newblock {\em Fixed-Point Algorithms for Inverse Problems in Science and
  Engineering}, pages 185--212, 2011.

\bibitem{cook2006camino}
P.~Cook, Y.~Bai, S.~Nedjati-Gilani, K.~Seunarine, M.~Hall, G.~Parker, and
  D.~Alexander.
\newblock Camino: Open-source diffusion-{MRI} reconstruction and processing.
\newblock In {\em 14th Scientific Meeting of the International Society for
  Magnetic Resonance in Medicine}, page 2759, 2006.

\bibitem{dey2006richardson}
N.~Dey, L.~Blanc-Feraud, C.~Zimmer, P.~Roux, Z.~Kam, J.-C. Olivo-Marin, and
  J.~Zerubia.
\newblock {R}ichardson-{L}ucy algorithm with total variation regularization for
  {3D} confocal microscope deconvolution.
\newblock {\em Microscopy Research and Technique}, 69:260--266, 2006.

\bibitem{doCarmo1992ri}
M.~do~Carmo.
\newblock {\em Riemannian Geometry}.
\newblock Birkh\"auser, Boston, 1992.

\bibitem{dong2009efficient}
Y.~Dong, M.~Hinterm{\"u}ller, and M.~Neri.
\newblock An efficient primal-dual method for $l^1$-$tv$ image restoration.
\newblock {\em SIAM Journal on Imaging Sciences}, 2:1168--1189, 2009.

\bibitem{drummond2002real}
T.~Drummond and R.~Cipolla.
\newblock Real-time visual tracking of complex structures.
\newblock {\em IEEE Transactions on Pattern Analysis and Machine Intelligence},
  24:932--946, 2002.

\bibitem{ferreira2002proximal}
O.~Ferreira and P.~Oliveira.
\newblock Proximal point algorithm on {R}iemannian manifolds.
\newblock {\em Optimization}, 51:257--270, 2002.

\bibitem{ferreira2013newton}
R.~Ferreira, J.~Xavier, J.~Costeira, and V.~Barroso.
\newblock Newton algorithms for riemannian distance related problems on
  connected locally symmetric manifolds.
\newblock {\em IEEE Journal of Selected Topics in Signal Processing},
  7:634--645, 2013.

\bibitem{fletcher2012}
P.~Fletcher.
\newblock Geodesic regression and the theory of least squares on {R}iemannian
  manifolds.
\newblock {\em International Journal of Computer Vision}, 105:171--185, 2013.

\bibitem{fletcher2007riemannian}
P.~Fletcher and S.~Joshi.
\newblock Riemannian geometry for the statistical analysis of diffusion tensor
  data.
\newblock {\em Signal Processing}, 87:250--262, 2007.

\bibitem{fletcher2004principal}
P.~Fletcher, C.~Lu, S.~Pizer, and S.~Joshi.
\newblock Principal geodesic analysis for the study of nonlinear statistics of
  shape.
\newblock {\em IEEE Transactions on Medical Imaging}, 23:995--1005, 2004.

\bibitem{ganeiber2012estimation}
A.~Ganeiber.
\newblock {\em Estimation and simulation in directional and statistical shape
  models}.
\newblock PhD thesis, University of Leeds, 2012.

\bibitem{getreuer2012rudin}
P.~Getreuer.
\newblock {Rudin-Osher-Fatemi total variation denoising using split Bregman}.
\newblock {\em {Image Processing On Line}}, 2012.
\newblock \url{http://dx.doi.org/10.5201/ipol.2012.g-tvd}.

\bibitem{goldstein2009split}
T.~Goldstein and S.~Osher.
\newblock The split {B}regman method for {$L^{1}$}-regularized problems.
\newblock {\em SIAM Journal on Imaging Sciences}, 2:323--343, 2009.

\bibitem{gousseau2001natural}
Y.~Gousseau and J.-M. Morel.
\newblock Are natural images of bounded variation?
\newblock {\em SIAM Journal on Mathematical Analysis}, 33:634--648, 2001.

\bibitem{green2006bayesian}
P.~Green and K.~Mardia.
\newblock {B}ayesian alignment using hierarchical models, with applications in
  protein bioinformatics.
\newblock {\em Biometrika}, 93:235--254, 2006.

\bibitem{grohs2013optimal}
P.~Grohs, H.~Hardering, and O.~Sander.
\newblock Optimal a priori discretization error bounds for geodesic finite
  elements.
\newblock Preprint, 2013.

\bibitem{grohs2009interpolatory}
P.~Grohs and J.~Wallner.
\newblock Interpolatory wavelets for manifold-valued data.
\newblock {\em Applied and Computational Harmonic Analysis}, 27:325--333, 2009.

\bibitem{hoff2009simulation}
P.~Hoff.
\newblock Simulation of the matrix {B}ingham-von {M}ises-{F}isher distribution,
  with applications to multivariate and relational data.
\newblock {\em Journal of Computational and Graphical Statistics}, 18:438--456,
  2009.

\bibitem{johansen2009diffusion}
H.~Johansen-Berg and T.~Behrens.
\newblock {\em Diffusion MRI: From quantitative measurement to in-vivo
  neuroanatomy}.
\newblock Academic Press, London, 2009.

\bibitem{jung2010randomgenerator}
S.~Jung.
\newblock Random number generation from von {M}ises--{F}isher distribution,
  2010.
\newblock \url{http://www.unc.edu/~sungkyu/manifolds/randvonMisesFisherm.m}.

\bibitem{karcher1977riemannian}
H.~Karcher.
\newblock Riemannian center of mass and mollifier smoothing.
\newblock {\em Communications on Pure and Applied Mathematics}, 30:509--541,
  1977.

\bibitem{kendall1990probability}
W.~Kendall.
\newblock Probability, convexity, and harmonic maps with small image {I}:
  uniqueness and fine existence.
\newblock {\em Proceedings of the London Mathematical Society}, 3:371--406,
  1990.

\bibitem{khatri1977mises}
C.~Khatri and K.~Mardia.
\newblock The von {M}ises-{F}isher matrix distribution in orientation
  statistics.
\newblock {\em Journal of the Royal Statistical Society. Series B}, 39:95--106,
  1977.

\bibitem{kimmel2002orientation}
R.~Kimmel and N.~Sochen.
\newblock Orientation diffusion or how to comb a porcupine.
\newblock {\em Journal of Visual Communication and Image Representation},
  13:238--248, 2002.

\bibitem{le2001diffusion}
D.~Le~Bihan, J.-F. Mangin, C.~Poupon, C.~Clark, S.~Pappata, N.~Molko, and
  H.~Chabriat.
\newblock Diffusion tensor imaging: Concepts and applications.
\newblock {\em Journal of Magnetic Resonance Imaging}, 13:534--546, 2001.

\bibitem{massonnet1998radar}
D.~Massonnet and K.~Feigl.
\newblock Radar interferometry and its application to changes in the earth's
  surface.
\newblock {\em Reviews of Geophysics}, 36:441--500, 1998.

\bibitem{moakher2002means}
M.~Moakher.
\newblock Means and averaging in the group of rotations.
\newblock {\em SIAM journal on matrix analysis and applications}, 24:1--16,
  2002.

\bibitem{moreau1962fonctions}
J.-J. Moreau.
\newblock Fonctions convexes duales et points proximaux dans un espace
  hilbertien.
\newblock {\em CR Acad. Sci. Paris S{\'e}r. A Math.}, 255:2897--2899, 1962.

\bibitem{nikolova2002minimizers}
M.~Nikolova.
\newblock Minimizers of cost-functions involving nonsmooth data-fidelity terms.
  {A}pplication to the processing of outliers.
\newblock {\em SIAM Journal on Numerical Analysis}, 40:965--994, 2002.

\bibitem{nikolova2004variational}
M.~Nikolova.
\newblock A variational approach to remove outliers and impulse noise.
\newblock {\em Journal of Mathematical Imaging and Vision}, 20:99--120, 2004.

\bibitem{nikolova2008efficient}
M.~Nikolova, M.~Ng, S.~Zhang, and W.~Ching.
\newblock Efficient reconstruction of piecewise constant images using nonsmooth
  nonconvex minimization.
\newblock {\em SIAM Journal on Imaging Sciences}, 1:2--25, 2008.

\bibitem{nikolova2013fast}
M.~Nikolova and G.~Steidl.
\newblock Fast hue and range preserving histogram specification: Theory and new
  algorithms for color image enhancement.
\newblock Preprint, 2013.

\bibitem{pennec2006riemannian}
X.~Pennec, P.~Fillard, and N.~Ayache.
\newblock A {R}iemannian framework for tensor computing.
\newblock {\em International Journal of Computer Vision}, 66:41--66, 2006.

\bibitem{petrushev1999nonlinear}
P.~Petrushev, A.~Cohen, H.~Xu, and R.~DeVore.
\newblock Nonlinear approximation and the space {$BV$}({$\mathbb{R}^2$}).
\newblock {\em American Journal of Mathematics}, 121:587--628, 1999.

\bibitem{rezakhaniha2012experimental}
R.~Rezakhaniha, A.~Agianniotis, J.~Schrauwen, A.~Griffa, D.~Sage, C.~Bouten,
  F.~Van~de Vosse, M.~Unser, and N.~Stergiopulos.
\newblock Experimental investigation of collagen waviness and orientation in
  the arterial adventitia using confocal laser scanning microscopy.
\newblock {\em Biomechanics and Modeling in Mechanobiology}, 11:461--473, 2012.

\bibitem{rocca1997overview}
F.~Rocca, C.~Prati, and A.~Ferretti.
\newblock An overview of {SAR} interferometry.
\newblock In {\em Proceedings of the $3$rd {ERS} Symposium on Space at the
  Service of our Environment, Florence}, 1997.
\newblock
  \url{http://earth.esa.int/workshops/ers97/program-details/speeches/rocca-et-al}.

\bibitem{rudin1992nonlinear}
L.~Rudin, S.~Osher, and E.~Fatemi.
\newblock Nonlinear total variation based noise removal algorithms.
\newblock {\em Physica D: Nonlinear Phenomena}, 60:259--268, 1992.

\bibitem{setzer2012vector}
S.~Setzer, G.~Steidl, and T.~Teuber.
\newblock On vector and matrix median computation.
\newblock {\em Journal of Computational and Applied Mathematics},
  236:2200--2222, 2012.

\bibitem{steidl2004equivalence}
G.~Steidl, J.~Weickert, T.~Brox, P.~Mr{\'a}zek, and M.~Welk.
\newblock On the equivalence of soft wavelet shrinkage, total variation
  diffusion, total variation regularization, and {SIDE}s.
\newblock {\em SIAM Journal on Numerical Analysis}, 42:686--713, 2004.

\bibitem{strekalovskiy2011total}
E.~Strekalovskiy and D.~Cremers.
\newblock Total variation for cyclic structures: Convex relaxation and
  efficient minimization.
\newblock In {\em IEEE Conference on Computer Vision and Pattern Recognition},
  pages 1905--1911, 2011.

\bibitem{strong2003edge}
D.~Strong and T.~Chan.
\newblock Edge-preserving and scale-dependent properties of total variation
  regularization.
\newblock {\em Inverse problems}, 19:S165, 2003.

\bibitem{Sturm2003he}
K.-T. Sturm.
\newblock Probability measures on metric spaces of nonpositive curvature.
\newblock In {\em Heat kernels and analysis on manifolds, graphs, and metric
  spaces}, volume 338 of {\em Contemp. Math.}, pages 357--390. American
  Mathematical Society, Providence, 2003.

\bibitem{tron2012riemannian}
R.~Tron, B.~Afsari, and R.~Vidal.
\newblock Riemannian consensus for manifolds with bounded curvature.
\newblock {\em IEEE Transactions on Automatic Control}, 58:921--934, 2013.

\bibitem{tschumperle2001diffusion}
D.~Tschumperl{\'e} and R.~Deriche.
\newblock Diffusion tensor regularization with constraints preservation.
\newblock In {\em Proceedings of the IEEE Conference on Computer Vision and
  Pattern Recognition}, pages I948--I953, 2001.

\bibitem{unser2013introduction}
M.~Unser and P.~Tafti.
\newblock {\em An introduction to sparse stochastic processes}.
\newblock 2013, to appear.

\bibitem{rahman2005multiscale}
I.~Ur~Rahman, I.~Drori, V.~Stodden, D.~Donoho, and P.~Schr{\"o}der.
\newblock Multiscale representations for manifold-valued data.
\newblock {\em Multiscale Modeling and Simulation}, 4:1201--1232, 2005.

\bibitem{vese2002numerical}
L.~Vese and S.~Osher.
\newblock Numerical methods for p-harmonic flows and applications to image
  processing.
\newblock {\em SIAM Journal on Numerical Analysis}, 40:2085--2104, 2002.

\bibitem{wallner2005convergence}
J.~Wallner and N.~Dyn.
\newblock Convergence and {$C^1$} analysis of subdivision schemes on manifolds
  by proximity.
\newblock {\em Computer Aided Geometric Design}, 22:593--622, 2005.

\bibitem{weinmann2012interpolatory}
A.~Weinmann.
\newblock Interpolatory multiscale representation for functions between
  manifolds.
\newblock {\em SIAM Journal on Mathematical Analysis}, 44:162--191, 2012.

\bibitem{wood1994simulation}
A.~Wood.
\newblock Simulation of the von {M}ises-{F}isher distribution.
\newblock {\em Communications in Statistics-Simulation and Computation},
  23:157--164, 1994.

\bibitem{yang2010fast}
J.~Yang, Y.~Zhang, and W.~Yin.
\newblock A fast alternating direction method for {TVL1-L2} signal
  reconstruction from partial {F}ourier data.
\newblock {\em IEEE Journal of Selected Topics in Signal Processing},
  4:288--297, 2010.

\bibitem{zach2007duality}
C.~Zach, T.~Pock, and H.~Bischof.
\newblock A duality based approach for realtime {TV}-{$L^1$} optical flow.
\newblock In {\em Pattern Recognition}, pages 214--223. Springer, 2007.

\end{thebibliography}

\end{document}